\documentclass[a4paper,8pt]{article}

\usepackage{authblk}
\usepackage{cite}
\usepackage[british]{babel}
\usepackage[utf8]{inputenc}
\usepackage{xcolor}
\usepackage{graphicx,subfigure}
\usepackage{epstopdf}
    \graphicspath{{./}{figures/}}
\usepackage[colorlinks=true,linkcolor=blue,citecolor=blue]{hyperref}
\usepackage{amsfonts,amsmath,amsthm,dsfont}

\newtheorem{theorem}{Theorem}[section]

\newtheorem{corollary}[theorem]{Corollary}
\theoremstyle{remark}\newtheorem{remark}[theorem]{Remark}

\newcommand{\be}{\begin{equation}}
\newcommand{\ee}{\end{equation}}
\newcommand{\z}{\mathbf{z}}

\newcommand{\R}{\mathbb{R}}
\newcommand{\x}{\mathbf{x}}

\newcommand{\fer}[1]{(\ref{#1})}
\allowdisplaybreaks

\begin{document}
\title{Impact of interaction forces in first order many-agent systems for swarm manufacturing}

\author[1,3]{Ferdinando Auricchio \thanks{\texttt{ferdinando.auricchio@unipv.it}}} 
\author[1]{Massimo Carraturo \thanks{\texttt{massimo.carraturo@unipv.it}}}
\author[2,3]{Giuseppe Toscani \thanks{\texttt{giuseppe.toscani@unipv.it}}}
\author[2]{Mattia Zanella \thanks{\texttt{mattia.zanella@unipv.it}}}

\affil[1]{Dept. of Civil Engineering and Architecture, University of Pavia} 
\affil[2]{Dept. of Mathematics ``F. Casorati'', University of Pavia} 
\affil[3]{IMATI "E. Magenes", CNR, Pavia}

\date{}

\maketitle

\begin{abstract}
We study the large time behavior of a system of interacting agents modeling the relaxation of a large swarm of robots, whose task is to uniformly cover a portion of the domain by communicating with each other in terms of their distance.  To this end, we generalize a related result for a Fokker-Planck-type model with a nonlocal discontinuous drift and constant diffusion, recently introduced by three of the authors, of which the steady distribution is explicitly computable. For this new nonlocal Fokker-Planck equation, existence, uniqueness and positivity of a global solution are proven, together with precise equilibration rates of the solution towards its quasi-stationary distribution. Numerical experiments are designed to verify the theoretical findings and explore possible extensions to more complex scenarios. 
\medskip


\end{abstract}

\tableofcontents
\section{Introduction}

In recent years the self-organizing features of large systems of interacting particles, and their application to social and life sciences, have been the subject of a huge number of research in the mathematical community. Without intending to review the huge literature on these topics, we point the interested reader to \cite{AP,C,HJKPZ,HKSF,K_etal,MT,TZ,V} and the references therein for an introduction. 

Among other approaches, the behavior of interacting particles can be conveniently captured through the powerful methodology of statistical physics, in particular resorting to collisional kinetic-type equations and their grazing limits \cite{BCC,CFRT,DM,HT,PT}. 

One of the recent fields of applications of agent-based modelling to industrial processes in which self-organizing systems play a major rule, is the fascinating field of swarm robotics, in which a swarm of interacting manufacturing agents is designed to produce complex components \cite{Ac,CH,D_etal,H,HW,Ox}.  Mathematical models can help the designer in predicting the swarm behavior, given a predefined environment and a set of process constraints \cite{A}.   An important feature of the underlying mathematical models is to be able to characterize the force of interaction between pairs of particles/agents in order to correctly capture the possible lack of information of the agents in achieving the desired goal.  To this end,  we also mention possible approaches having roots in the control literature \cite{Ahn2,CFPT,CKPP,COT}.

In this work we aim to shed light on some of the aforementioned problems starting from the kinetic model presented in \cite{ATZ}, where a large swarm of agents capable to spread uniformly over the surface of a isotropic domain $D \subset \mathbb R^d$ has been introduced and studied in terms of a suitable Fokker--Planck type equation  with discontinuous drift, which senses the distance from the boundaries of $D$. This simple task can indeed be interpreted as the deposition of a single layer in standard additive manufacturing processes. We mention that possible extensions of the presented setting may focus on general target domains on manifolds, see e.g. \cite{Ahn_etal}. 

At variance with the classical drift function considered in \cite{ATZ}, to model the strength of interactions among particles, in the present contribution we introduce a nonlocal drift operator depending on a symmetric interaction function weighting the influence of the swarm positions' distribution on a particle.  In agreement with the analysis of \cite{ATZ2}, it will be shown that the  process can be fruitfully  analyzed by resorting to  a novel Fokker-Planck-type equation, which, while possessing the same steady profile, is characterized by a linear drift and a suitable time-dependent diffusion coefficient. The time-dependence of the diffusion coefficient is closely related to the main feature of the  model, and precisely to the presence of a limited information on the direction of motion, and therefore on the position of the target domain.

For this new nonlocal Fokker--Planck equation, existence, uniqueness and positivity of a global solution are proven, together with equilibration rates of the solution towards its quasi-stationary distribution. In particular, it is shown that the limited information among particles slows down of the speed of convergence towards equilibrium of the solution. The identification of the analytical steady state of the model depends on the existence of quasi-stationary solutions of the  Fokker-Planck equation which converge in time towards a global steady continuous distribution with unitary mass. In the following we will prove such convergence in the case of constant communication rate between particles. Furthermore, in the last section we will give computational insight on the case where the communication strength depends by the relative distance between agents. This case is of particular relevance for applications, see e.g. \cite{Cavanna,CS,DOB}. 

In more detail, the paper is organized as follows. In Section \ref{sect:2} we introduce the Fokker- Planck-type model for manufacturing, mimicking the interaction of a system of agents with a given portion of the domain, as well as the interactions between particles in terms of their distance. We then rigorously study the structure of the equilibrium profile together with its relevant features in one and two dimensions. Section \ref{sect:3} will be devoted to the study of the convergence of the solution towards equilibrium. To this aim we will resort to a new Fokker– Planck type equation possessing the same steady state but different drift and diffusion operators. Explicit results are established in the 1D case together with trends characterizing convergence. Finally, in Section \ref{sect:4} we propose several numerical examples in 1D and 2D to test both the features of the model and the convergence rates.

\section{Fokker-Planck models of swarms}\label{sect:2}

We consider a system of $N\gg 1$ particles interacting with each other and  with a space domain $D \subset \mathbb R^d$. To simplify our analysis, in the following we will assume that $D$ is a $d$-dimensional sphere centered in $\x_0$ with radius $\delta>0$, i.e. $D = \{\x \in \mathbb R^d: |\x-\x_0|\le \delta\}$, being $|\cdot|$ the Euclidean distance between $\x, \mathbf y \in \mathbb R^d$. 

We assume that the system of particles is such that each particle modifies its position interacting with all the other particles. Furthermore, we assume that the particles sense the direction of motion towards the center of the sphere $D$ and, once inside $D$, they start to randomly explore the target domain. 

Let $f(\x,t)\, d\x$ denote the probability of finding a particle in the elementary volume $d\x$ around the point $\x \in \mathbb R^d$ at time $t \ge 0$. The mesoscopic model that translates this dynamics can be described by a Fokker-Planck equation with constant diffusion and discontinuous, time-dependent drift, which can be appropriately written in divergence form as follows 
\begin{equation}
\label{eq:FP}
\partial_t f(\x,t) = \nabla_\x \cdot \left[ \mathcal B[f](\x,t)\mathds{1}_{D^c}(\x)f(\x,t) + \sigma^2 \nabla_\x f(\x,t) \right],
\end{equation}
where  $\sigma^2>0$ is the positive constant coefficient of diffusion, and 
\begin{equation}
\label{eq:drift1}
\mathcal B[f](\x,t) = \lambda (\x-\x_0) + \mu \int_{\mathbb R^d} P(\x,\mathbf y)(\x - \mathbf y)f(\mathbf y,t)d\mathbf y,
\end{equation}
is the drift coefficient. In \fer{eq:drift1} the constants  $\lambda,\mu\ge 0$, such that $\lambda + \mu = 1$,  denote the intensities of the classical drift, and respectively of the  communication between particles. Moreover  $P(\cdot,\cdot) \ge 0$ is a symmetric interaction function weighting the influence on a particle in $\mathbf x$ of all the other particles in terms of their distance with $\x$, i.e. $P(\x,\mathbf y) = P(\mathbf y,\x)$. In equation \fer{eq:FP} we indicated with $\mathds{1}_{D^c}: \mathbb R^d \rightarrow \{0,1\}$ the indicator function of the complement set of the spherical domain $D$, denoted by $D^c$,  that is
\[
\mathds{1}_{D^c}(\x) = 
\begin{cases}
1 & \x \in D^c \\
0 & \x \notin D^c.
\end{cases}
\] 
According to equation \fer{eq:FP},   particles of the swarm move subject to the simultaneous presence of drift and diffusion unless they are in the target domain $D$ where only the diffusion operator survives. 

In \cite{ATZ} it has been observed how, in the case $\mu=0$, namely in absence of communication among particles, the steady state of unit mass of the resulting Fokker-Planck equation is the unique solution of the differential equation 
\begin{equation}
\label{eq:diff}
\sigma^2 \nabla_\x f(\x,t) = - (\x-\x_0) \mathds{1}_{D^c}(\x) f(\x,t), 
\end{equation}
 given by 
\begin{equation}
\label{eq:finfty_1}
f^\infty(\x) = 
\begin{cases}
\dfrac{m_1}{(2\pi \sigma^2)^{d/2}} \exp\left\{-\dfrac{|\x-\x_0|^2}{2\sigma^2}\right\} & |\x-\x_0|\ge \delta,\\
m_2 \left(\dfrac{\delta^d \pi^{d/2}}{\Gamma(d/2 + 1)} \right)^{-1} & |\x-\x_0|< \delta.
\end{cases}
\end{equation}
In \fer{eq:finfty_1} the constants $m_1,m_2>0$ are determined by imposing unitary total mass and the continuity of the steady state at the boundary of the target domain $D\subset \mathbb R^d$. These conditions correspond to the linear system
\[
\dfrac{m_1}{(2\pi\sigma^2)^{d/2}} \int_{|\x-\x_0|\ge \delta} \exp\left\{-\dfrac{|\x-\x_0|^2}{2\sigma^2}\right\} d\x+ m_2 = 1
\]
and
\[
\lim_{|\x-\x_0|\to \delta} \dfrac{m_1}{(2\pi\sigma^2)^{d/2}}\exp\left\{-\dfrac{|\x-\x_0|^2}{2\sigma^2}\right\} = \dfrac{m_2}{\delta^d \pi^{d/2}}\Gamma(\delta/2 + 1). 
\]
In \cite{ATZ} it has been  observed that the constants $m_1,m_2$ are uniquely determined by the data $\delta,\sigma^2>0$. Furthermore, it can be observed that $m_1,m_2$ are strictly positive, coherently with the fact that  \eqref{eq:finfty_1} is a distribution, see \cite{ATZ}. Precise computations can be done  in the case of 1D target domain $D = \left\{ x \in \mathbb R: |x-x_0|\le \delta\right\}\subset \mathbb R$,  where the system takes the form 
\begin{equation}
\label{eq:system1}
\begin{cases}
m_1 \left( 1-\textrm{erf}\left( \dfrac{\delta}{\sqrt{2\sigma^2}}\right)\right) + m_2 = 1 \\
\dfrac{m_1}{\sqrt{2\pi\sigma^2}}\exp\left\{-\dfrac{\delta^2}{2\sigma^2}\right\}-\dfrac{m_2}{2\delta} = 0,
\end{cases}
\end{equation}
and in the 2D case, where $D = \left\{ \x \in \mathbb R^2: |\x-\x_0|\le \delta\right\}\subset \mathbb R^2$, and the system reads 
 \begin{equation}
 \label{eq:system2}
 \begin{cases}
 m_1\exp\left\{-\dfrac{\delta^2}{2\sigma^2}\right\} +m_2 = 1 \\
 \dfrac{m_1}{2\sigma^2}\exp\left\{-\dfrac{\delta^2}{2\sigma^2}\right\} -\dfrac{m_2}{\delta^2}=0. 
 \end{cases}
 \end{equation}
In both cases, we may fix $m_2>0$ and $\delta>0$ and we may prove that there exist $m_1,\sigma^2>0$ solutions to \eqref{eq:system1} or \eqref{eq:system2} and that these are unique. 

The resulting steady state is a continuous function given by the weighted combination of a Gaussian density outside $D$ and a uniform density inside $D$. In Figure \ref{fig:1} we depict the steady state \eqref{eq:finfty_1} for several choices of $m_2$ and $\delta = 0.5$ for the 1D case and $\delta = 1/ \sqrt{\pi}$ for the 2D case. 

\begin{figure}
\centering
\includegraphics[scale = 0.3]{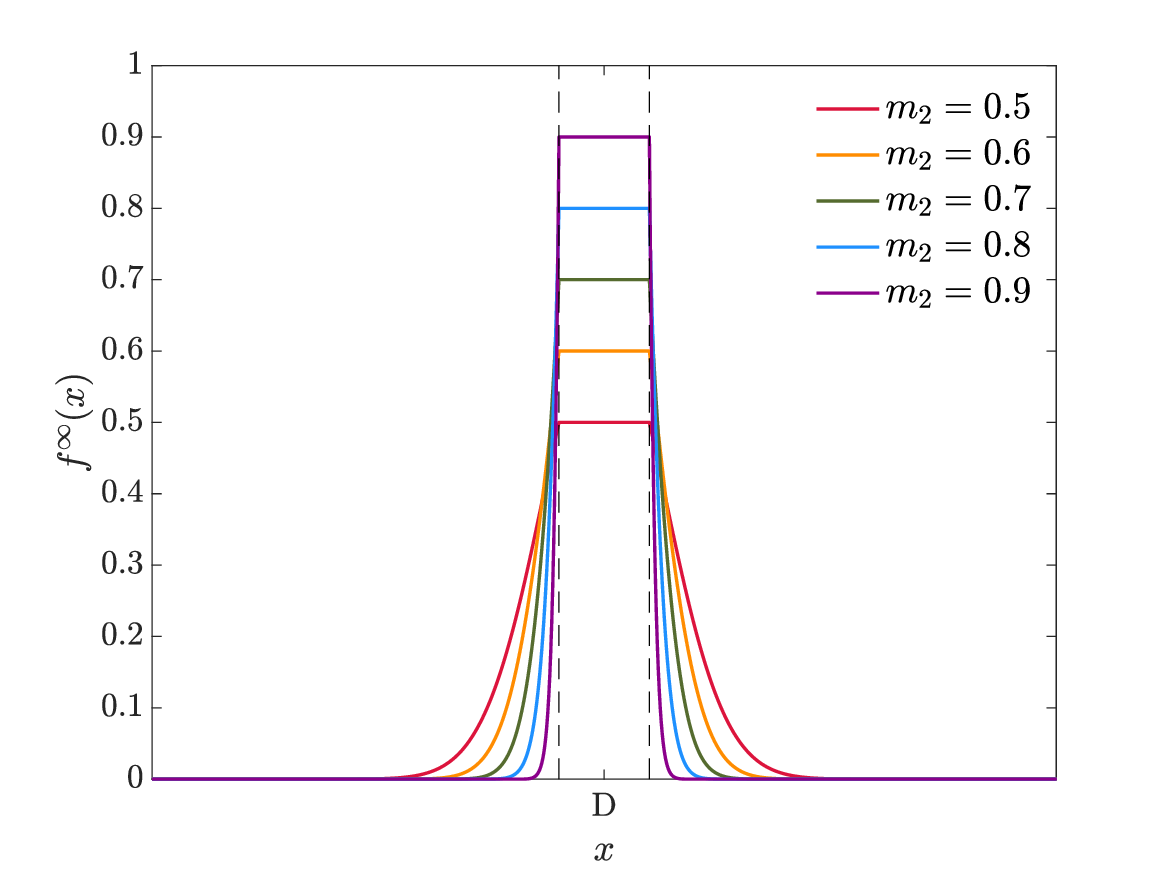}
\includegraphics[scale = 0.3]{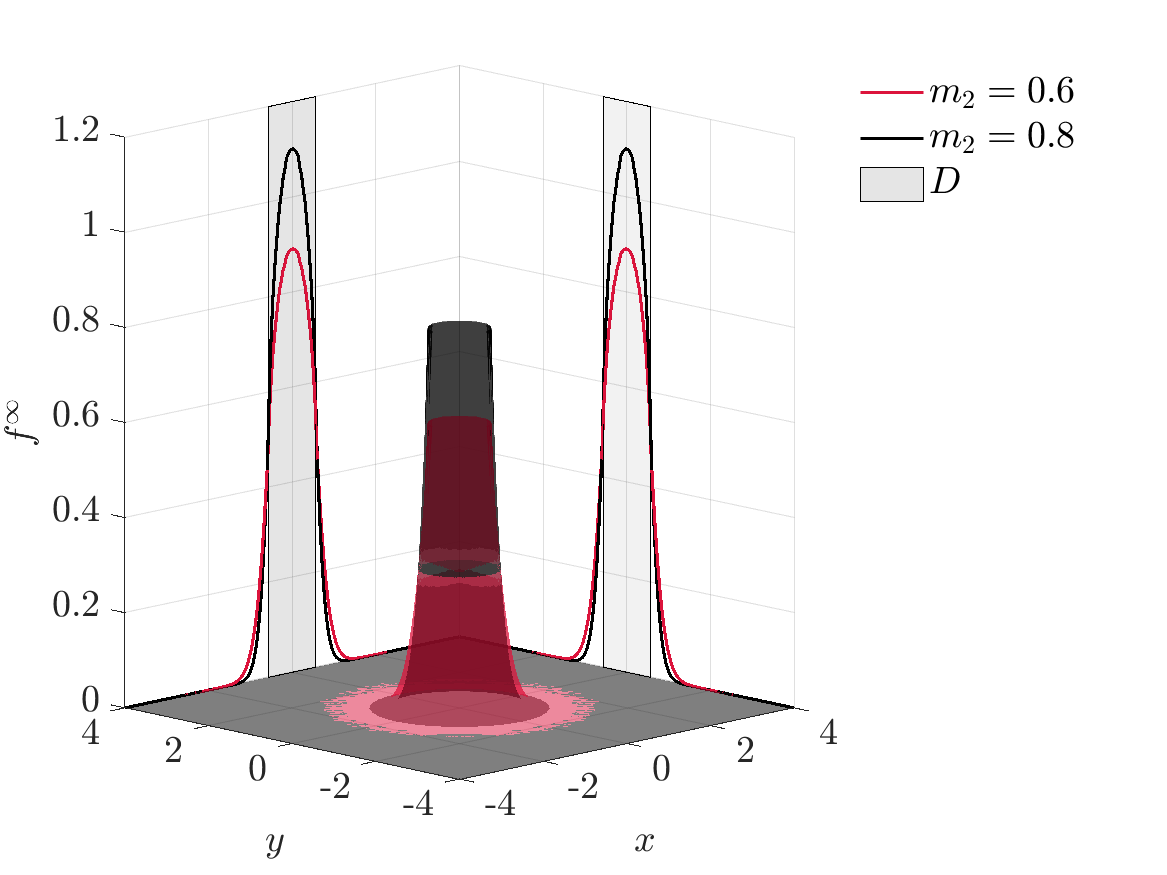}
\caption{Left: steady state distribution for the 1D problem with $\delta = 0.5$ and $x_0 = 0$. Right: steady state distribution for the 2D problem with $\delta = \sqrt{\pi^{-1}}$ and $\x_0 = (0,0)$. In both cases we fixed several values of $m_2>0$ and $\delta>0$.  }
\label{fig:1}
\end{figure}

 In the absence of communications between particles, In \cite{ATZ} it was further remarked that the equilibrium distribution $f^\infty(\x)$ defined in \eqref{eq:finfty_1} could be obtained by resorting to a surrogate Fokker-Planck equation characterized by a linear and continuous drift and a variable diffusion coefficient, given by
\begin{equation}
\label{eq:FPlambda1}
\partial_t f(\x,t) = \nabla_\x \cdot \left[(\x-\x_0)f(\x,t) + \nabla_\x (K(\x)f(\x,t)) \right],
\end{equation}
where the diffusion coefficient $K$ is 
\begin{equation}
\label{eq:kappa_lambda1}
K(\x) = 
\begin{cases}
\sigma^2 + \dfrac{\delta^2}{2} - \dfrac{|\x-\x_0|^2}{2} & |\x-\x_0|<\delta, \\
\sigma^2 & |\x-\x_0|\ge\delta,
\end{cases}
\end{equation}
 The interesting feature of the Fokker--Planck equation \fer{eq:FPlambda1} is related to the fact that, as observed in  \cite{ATZ},   convergence to equilibrium of the solution at a polynomial rate follows by resorting to entropy decay, while the same can not be directly obtained  for the original Fokker--Planck equation \fer{eq:FPlambda1} in reason of the lack of convexity of the potential characterizing the discontinuous drift function. 

Owing to the same idea developed in \cite{ATZ},  for any given pair of positive constants $\lambda,\mu\in [0,1]$, such that $\lambda +\mu=1$,  a Fokker-Planck model with continuous drift can be obtained by solving the following differential equations
\begin{equation}
\label{eq:cond}
\begin{cases}
\mathcal B[f](\x,t) + \nabla_\x \kappa(\x,t) = 0, & |\x-\x_0|\le \delta, \\
\nabla_\x \kappa(\x,t) = 0, & |\x-\x_0|\ge \delta,
\end{cases}
\end{equation}
together with continuity at the interface of the domain $D \subset \mathbb R$. 

Therefore, if a $\kappa(\cdot,\cdot)$ solution to \eqref{eq:cond} exists, we can substitute the action of the swarm, as given by the model \eqref{eq:FP}, with  the Fokker-Planck equation with nonlocal diffusion and linear drift 
\begin{equation}
\label{eq:FP2}
\partial_t f(\x,t) = \nabla_\x \cdot \left[\mathcal B[f](\x,t) f(\x,t) + \nabla_\x (\kappa(\x,t)f(\x,t) )\right].
\end{equation}

At variance with \eqref{eq:FP}-\eqref{eq:drift1}, the  Fokker-Planck equation \fer{eq:FP2} describes a swarm of particles such that each particle senses the direction of motion together with the location of other particles, moves towards the center of the sphere $D$ trying to relax towards the mean position of the swarm, and it starts to randomly explore $D$ by adapting its diffusive behavior  to its distance from the center $\x_0\in D$ of the domain. 
Clearly, even if the steady states of unit mass of the Fokker-Planck models \eqref{eq:FP} and \eqref{eq:FP2} are equal,  the solutions  may differ in the transient regime. 

\begin{remark}
It is important to remark that the variable coefficient of diffusion  $\kappa(\x,t)$ defined in \eqref{eq:FP_simp_diffusion} is positive and uniformly bounded, i.e. $0<c\le \kappa(\x,t)\le C$ where
\[
c = \min_{\x \in \mathbb R^d,t\ge0}\kappa(\x,t) = \sigma^2, \qquad C = \max_{\x \in \mathbb R^d, t\ge0}\kappa(\x,t) = \sigma^2+\dfrac{\delta^2}{2}.
\]
\end{remark}

\subsection{Qualitative properties of the model}\label{sect:21}
Let us suppose that a solution to \eqref{eq:cond} exists. Then, the evolution of the main moments of the distribution $f$ are obtained from the weak formulation of \eqref{eq:FP2} expressed by 
\begin{equation}
\label{eq:weak}
\begin{split}
&\dfrac{d}{dt} \int_{\mathbb R^d} \varphi(\x) f(\x,t)d\x \\
&\quad = \int_{\mathbb R^d} \varphi(\x) \nabla_\x \cdot \left[ \mathcal B[f](\x,t)f(\x,t) + \nabla_\x ( \kappa(\x,t)f(\x,t) ) \right]d\x,
\end{split}
\end{equation}
with $\varphi(\cdot) \in \mathcal C_0$ a smooth test function. By choosing $\varphi(\x) = 1$,  we obtain that the mass of the system of particles is preserved in time. Next,  the choice $\varphi(\x) = \x$ gives the evolution of the mean value  $\mathbf u(t)$, where
\[
\mathbf u(t) = \int_{\mathbb R^d}\x f(\x,t)\, d\x.
\]
Since the interaction function $P(\cdot,\cdot)\ge 0$ is symmetric, we obtain
\[
\begin{split}
\dfrac{d}{dt} \mathbf u(t) &= - \int_{\mathbb R^d} \mathcal B[f](\x,t) f(\x,t)d\x -\int_{\mathbb R^d}  \nabla_\x \cdot(\kappa(\x,t)f(\x,t))d\x, \\
&=  - \lambda (\mathbf u(t)-\x_0) + \mu \underbrace{\int_{\mathbb R^d} \int_{\mathbb R^d} P(\x,\mathbf y)(\x-\mathbf y)f(\mathbf y,t)f(\x,t)d\mathbf y \,d \x}_{=0},
\end{split}\]
Therefore, the evolution of the mean position of the swarm only depends on the constant $\lambda \in [0,1]$  through the formula
\begin{equation}
\label{eq:uevo}
\mathbf u(t) = \x_0+(\mathbf u(0) - \x_0)e^{-\lambda t}.
\end{equation}
Therefore,  $\mathbf u(t) \to \x_0$ for $t\to +\infty$ for all $\lambda\in (0,1]$,  independently on $\mu>0$ and on the form of a symmetric communication function $P\ge0$. Note however that an explicit expression for the mean could not be obtained for the solution to the Fokker--Planck equation \fer{eq:FP} in reason of the discontinuity of the drift coefficient.

Boundedness of energy can be easily obtained by studying the evolution of the second order moment, that is by choosing  $\varphi(\x) = |\x|^2/2$.  If
\[
 E(t) = \int_{\mathbb R^d}\frac{|\x|^2}{2} f(\x,t)\, d\x,
 \]
 we obtain
\[
\begin{split}
\dfrac{d}{dt} \int_{\mathbb R^d}\dfrac{|x|^2}{2}f(\x,t)d\x &= - \int_{\mathbb R^d}\nabla_\x\dfrac{|\x|^2}{2}\cdot \mathcal B[f](\x,t)f(\mathbf y,t)d\mathbf y \\
&\quad + \int_{\mathbb R^d} \Delta_\x \dfrac{|\x|^2}{2} \kappa(\x,t)f(\x,t)d\x.
\end{split}\]
Therefore, since $\kappa$ uniformly bounded from above, we have
\begin{equation}
\label{eq:ener}
\begin{split}
\dfrac{dE(t)}{dt}\le &C-2\lambda E(t) + \lambda \x_0\cdot \mathbf{u}(t)\\
&\quad-\mu\int_{\mathbb R^{2d}}P(\x,\mathbf{y})\x\cdot(\x-\mathbf{y})f(\x,t)f(\mathbf{y},t)d\x d\mathbf{y} \\
\le&C - 2\lambda E(t) + \lambda \x_0\cdot \mathbf{u}(t) + \mu \mathbf{u}^2(t).
\end{split}
\end{equation}
Finally, since $\mathbf{u}(t)$, as given by  \eqref{eq:uevo} is uniformly bounded from above and below, we conclude. 

\subsection{The uniform interaction case}\label{sect:22}
A closer insight on the evolution of the solution to equation \eqref{eq:FP2} can be obtained in the simplified case $P\equiv 1$. As detailed in \cite{ATZ2}, in the case $P \equiv 1$, we get 
\begin{equation}
\label{eq:FP_simp_drift}
\mathcal B[f](\x,t) = \x -\tilde \x_0(t), \qquad \tilde \x_0(t) = \lambda \x_0 + \mu \mathbf u(t).
\end{equation}
 Hence, the action of the swarm described by the Fokker--Planck equation \eqref{eq:FP} may be fruitfully recast by resorting to the Fokker--Planck equation \eqref{eq:FP2} with a variable diffusion function of the form
\begin{equation}
\label{eq:FP_simp_diffusion}
\kappa(\x-\tilde{\x}_0(t))=\kappa(\x,t) = 
\begin{cases}
\sigma^2 + \dfrac{\delta^2}{2} - \dfrac{|\x-\tilde \x_0(t)|^2}{2} & |\x-\tilde \x_0(t)|<\delta, \\
\sigma^2 & |\x-\tilde \x_0(t)|\ge \delta. 
\end{cases}
\end{equation}
It is worth to remark that, since the evolution of the mean position $\mathbf{u}(t)$ is given by \eqref{eq:uevo},  in the limit $t \rightarrow +\infty$ the diffusion coefficient $\kappa(\x,t)$ in \eqref{eq:FP_simp_diffusion} converges uniformly to the diffusion coefficient $K(\x)$ defined  by \eqref{eq:kappa_lambda1}.  

 Having this property in mind, it is immediate to realize that, to recover results about existence, uniqueness and positivity of solutions to  the Fokker-Planck equation with uniform interactions 
\begin{equation}
\label{eq:FP_P1}
\partial_t f(\x,t) = \nabla_\x \cdot \left[(\x-\tilde{\x}_0)f(\x,t) + \nabla_\x(\kappa(\x,t)f(\x,t)) \right]
\end{equation}
with $\kappa(\x,t)$ defined in \eqref{eq:FP_simp_diffusion},   we can resort to an alternative equivalent formulation, which consists in  rewriting the equation in terms of $\mathbf{z} = \x - \tilde{\x}_0(t)\in \mathbb R$. Therefore, if $g(\z,t) = f(\z + \tilde{\x}_0(t),t)$, $g$ satisfies the equation
\begin{equation}
\label{eq:g}
\partial_t g(\z,t) + \mathbf{A}(t)\cdot\nabla_\z g(\z,t) = \nabla_\z \cdot \left[ \z g(\z,t) + \nabla_\z(K(\z)g(\z,t)) \right],
\end{equation}
where $\mathbf{A}(t) = \lambda\mu (\mathbf{u}(0)-\x_0)e^{-\lambda t}$ and $K$ has been defined in \eqref{eq:kappa_lambda1}. The above equation is composed of a pure transport operator and a drift-diffusion operator with time independent coefficients, and where $K\ge \sigma^2>0$  is uniformly bounded. 

As shown in \cite{ATZ2} in the one-dimensional case, existence and uniqueness of the solution to \fer{eq:FP_P1} in $\R^d$ can be obtained by resorting to Proposition 2 of Section 6 of the paper by Le Bris and Lions \cite{LLB}, 
concerned with  Fokker-Planck type equations with irregular and time-dependent drift and diffusion coefficients. The following result holds, see \cite{ATZ2,LLB}. 
\begin{theorem}
\label{th:1}
We consider the initial value problem 
\begin{equation}
\label{eq:FPth}
\partial_t f(\x,t) = \nabla_\x \cdot\left[\mathcal{B}[f](\x,t)f(\x,t) + \nabla_\x(\kappa(\x,t)f(\x,t)) \right],
\end{equation}
for each initial condition $f(\x,0) = f_0(\x) \in L^1 \cap L^\infty(\mathbb R^d)$ (resp. $L^2 \cap L^\infty(\mathbb R^d)$) and 
\[
\dfrac{\mathcal B[f]}{1+|\x|}\in L^1+L^\infty(\mathbb R^d), \qquad \dfrac{\kappa(\x,t)}{1+|\x|} \in L^2 \cap L^\infty(\mathbb R^d).
\]
If for any $t\ge 0$ that $\nabla_\x \cdot \mathcal B[f] \in L^\infty(\mathbb R^d)$, $\mathcal B[f] \in W^{1,1}_{\textrm{loc}}(\mathbb R^d)$, $\kappa(\x,t) \in W^{1,2}_{\textrm{loc}}(\mathbb R^d)$, then equation \eqref{eq:FPth} has a unique solution in the space
\[
\begin{split}
&f(\x,t) \in L^\infty([0,T],L^1\cap L^\infty(\mathbb R^d) \textrm{(resp. $L^2 \cap L^\infty(\mathbb R^d))$},\\
&\kappa(\x,t)\nabla_\x f(\x,t) \in L^2([0,T],L^2(\mathbb R^{d})).
\end{split}
\]
\end{theorem}
\begin{proof}
We point the interested reader to the proof of Proposition 2 in Section 6 of \cite{LLB}. Concerning the finiteness of \eqref{eq:cond_f0} we may observe that the solution to \eqref{eq:FPth}, as shown in Section \ref{sect:21}, inequality \eqref{eq:ener}, has bounded second order moment. Hence, we may argue as in \cite{Arkeryd} to conclude that $f|\log^-f|\le e^{-|\x|^2} + |\x|^2 f \in L^1(\mathbb R^d)$. Since the solution $f(\x,t) \in L^\infty(\mathbb R^d)$, it follows that $f(\x,t) \log^+f(\x,t) \in L^1(\mathbb R^d)$ 
\end{proof}

\begin{remark}
We observe that in the case $P\equiv 1$ we get $\mathcal B[f](\x,t) = \x-\tilde{\x}_0(t)$ and a diffusion function of the form \eqref{eq:FP_simp_diffusion}. Therefore, the above result holds for all time $t \in [0,T]$ being
\[
\dfrac{\x-\tilde{\x}_0}{1+|\x|} = \dfrac{1}{(1+|\x|)^{d+1}} + \dfrac{(\x-\tilde{\x}_0)(1+|\x|)^d-1}{(1+|\x|)^{d+1}} \in L^1([0,T], L^1 + L^\infty(\mathbb R^d)),
\]
and $\kappa(\x,t)$ uniformly bounded. 
\end{remark}

Moreover we have 
\begin{corollary}\label{cor:1}
Under the same hypotheses of Theorem \ref{th:1}, if 
 $f_0(\x)$ is a probability density function such that 
\begin{equation}
\label{eq:cond_f0}
\int_{\mathbb R^d} (1+|\x|^2 + \log f_0(\x))f_0(\x)d\x <+\infty,
\end{equation}
then, for all $t \in [0,T]$, the unique solution $f(\x,t)$ to the Fokker--Planck equation \fer{eq:FPth} is a probability density function, and
\[
\int_{\mathbb R^d} (1+|\x|^2 + \log f(\x,t))f(\x,t)d\x <+\infty.
\]
\end{corollary}
\begin{proof}
In view of the linearity of equation \fer{eq:g},  properties of the solution to equation \fer{eq:FPth} can  be shown by resorting to its equivalent formulation \fer{eq:g}, discretizing this equation via the classical splitting method  \cite{temam}, and subsequently applying to this discretization of the solution the classical Trotter's formula. 

Let us briefly recall the splitting method.  For any given time $T >0$ and $n \in \mathbb N$, we introduce a time discretization $t^k = k\Delta t$, $k \in[0,n]$, with $\Delta t = T/n>0$. Then we proceed by solving two separate problems in each time step as follows:
\begin{enumerate}
\item At time $t = 0$ we start from $g^0(\z) = g(\z,0)\ge 0$, $g^0 \in H^1(\mathbb R)$. 
\item For $t\in[t^k,t^{k+1}]$ we solve the Fokker-Planck step
\be\label{eq:FP-step}
\begin{split} 
&\partial_t g(\z,t) = \nabla_\z \cdot \left[ \z g(\z,t) + \nabla_\z(K(\z)g(\z,t)) \right],\\
& g(\z,t^k) = g^k(\z) 
\end{split}
\ee
\item The solution of the Fokker--Planck step at time $t^{k+1}$ is assumed as the initial value for the transport step in the same time interval $t\in[t^k,t^{k+1}]$.  This is  usually done by denoting $g(\z, t^{k+1} )= g^{k+1/2}(\z)$  
\item
For $t \in [t^{k},t^{k+1}]$ the transport step is subsequently solved by considering 
\be\label{eq:trans-step}
\begin{split}
&\partial_t g(\z,t) +\mathbf{A}(t)\cdot \nabla_\z g(\z,t)=0,\\
&g(\z, t^k) = g^{k+1/2}(\z). 
\end{split}
\ee
\end{enumerate}
The method clearly generates an approximation to the solution of problem \fer{eq:g}, say $g_n(\x,t)$, for which properties can be easily derived by resorting to well-known properties of the underlying linear operators, in our case transport and drift--diffusion, which are solved in sequence. For example, positivity is immediate to derive in view of the positivity properties of both the operators involved into the splitting. Next, existence and uniqueness of the solution to \fer{eq:g} allow to conclude, via Trotter's formula, that 
\[
\lim_{n\to \infty} g_n(\x,t) = g(\x,t) \ge 0,
\]
where $g(\x,t)$ is the solution to \fer{eq:g}. This shows positivity. 

Concerning the finiteness of \eqref{eq:cond_f0} we may observe that the solution to \eqref{eq:FPth}, as shown in Section \ref{sect:21}, inequality \eqref{eq:ener}, has bounded second order moment. Hence, we may argue as in \cite{Arkeryd} to conclude that $f(\x,t)|\log^-f(\x,t) |\le e^{-|\x|^2} + |\x|^2 f(x,t) \in L^1(\mathbb R^d)$. Since the solution $f(\x,t) \in L^\infty(\mathbb R^d)$, it follows that $f(\x,t) \log^+f(\x,t) \in L^1(\mathbb R^d)$ .
\end{proof}

\section{Large time behavior via entropy decay}\label{sect:3}
In this section we focus on the convergence of the solution to the Fokker--Planck equation \eqref{eq:FP2} to its equilibrium distribution, in the  uniform interaction case $P\equiv 1$. The study of convergence towards equilibrium of classical kinetic equations is a well-known problem, which is classically based on the study of the time-decay of entropy functionals.  We point the interested reader to \cite{OV,T1,TV} for an introduction. 

Once established existence, uniqueness and positivity of the solution to \eqref{eq:FP2} in \cite{ATZ2} the following one-dimensional result  was obtained

\begin{theorem}\label{th:2}
Let $f(x,t)$, $x \in \mathbb R$, be the unique solution to the initial value problem \eqref{eq:FPth} departing from an initial condition $f_0(x)$ such that the hypotheses in Corollary \ref{cor:1} are satisfied. Then $f(x,t)$ converges in $L^1(\mathbb R)$ towards the one-dimensional steady solution $f_\infty(x)$ defined in \eqref{eq:finfty_1} and the convergence rate is at least $o(t^{-1/2})$.
\end{theorem}

The proof of Theorem \ref{th:2} has been classically based on the rigorous time-decay of the relative entropy, and on a one-dimensional  Wirtinger type inequality proved in \cite{FPTT}.  Since a multi-dimensional version of this inequality is not available, the extension of the result in higher dimensions can not be directly concluded. 

However, various partial results can be easily shown to hold.
As in the one-dimensional case,  the Fokker-Planck equation \eqref{eq:FP2}  possesses a quasi-stationary solution, namely a solution, for a fixed time $t>0$, of the first-order differential equation
\[
\left[(\x-\tilde{\mathbf{x}}_0(t))+ \nabla_\x \kappa(\x,t)\right]f(\x,t)  + \kappa(\x,t)\nabla_\x f(\x,t)=0
\]
This solution, in the time-independent case $\mu =0$, coincides with \eqref{eq:finfty_1}, whereas, for $\mu>0$, it is given by
\begin{equation}
\label{eq:fq}
f_q(\x,t) = \begin{cases}
\dfrac{m_1}{(2\pi \sigma^2)^{d/2}} \exp\left\{-\dfrac{|\x-\tilde{\x}_0(t)|^2}{2\sigma^2}\right\} & |\x-\tilde{\x}_0(t)|\ge \delta,\\
m_2 \left(\dfrac{\delta^d \pi^{d/2}}{\Gamma(d/2 + 1)} \right)^{-1} & |\x-\tilde{\x}_0(t)|< \delta, 
\end{cases}
\end{equation}
where $\tilde{\mathbf{\x}}_0(t) =\lambda \x_0 + \mu \mathbf{u}(t)$.

Also, the time-decay (without rate) of the relative entropy between the solution  to the Fokker-Planck equation \eqref{eq:FP2} and its quasi-stationary solution \fer{eq:fq} can be rigorously proven since the solution to \fer{eq:FP2} satisfies a weak maximum principle in any bounded domain of $\R^d$, and Theorem $7$ of \cite{FPTT} can be applied. 

The weak maximum principle can be obtained by resorting to the equivalent formulation \fer{eq:g}, along the line of the proof of Corollary \ref{cor:1}.

Since the  equilibrium state
\begin{equation}
\label{eq:finfty_g}
g^\infty(\x) = 
\begin{cases}
\dfrac{m_1}{(2\pi \sigma^2)^{d/2}} \exp\left\{-\dfrac{|\z|^2}{2\sigma^2}\right\} & |\z|\ge \delta,\\
m_2 \left(\dfrac{\delta^d \pi^{d/2}}{\Gamma(d/2 + 1)} \right)^{-1} & |\z|< \delta.
\end{cases}
\end{equation}
is a steady state of the drift--diffusion step \fer{eq:FP-step}, we can use, in any fixed bounded domain with suitably boundary conditions,  its adjoint form
\be\label{eq:FP-stepA}
\begin{split} 
&\partial_t G(\z,t) =  K(\z)\Delta_\z G(\z,t) -  \nabla_\z \cdot  \z G(\z,t) ,\\
& G(\z,t^k) = G^k(\z) 
\end{split}
\ee
where $G(\z,t) = g(\z,t)/g^\infty(\z)$. Likewise, in the transport step, it is immediate to conclude that, in view of the expression of $g^\infty(\z)$, we can obtain the evolution equation
\be\label{eq:FP-stepB}
\begin{split} 
&\partial_t G(\z,t) +\mathbf{A}(t)\cdot\left[ \nabla_\z G(\z,t) - G(\z,t) \mathds{1}(|\z|\ge \delta) \frac\z\sigma  \right]=0,\\
&G(\z, t^k) = G^{k+1/2}(\z)
\end{split}
\ee
Now,  the uniform boundedness of the diffusion coefficient $K(\z)$ ensures that the solution to the step \fer{eq:FP-stepA} satisfies  the maximum principle.  Likewise, integrating equation \fer{eq:FP-stepB} along characteristics shows that its solution satisfies, in any fixed bounded domain, a weak maximum principle, in the sense that,  if the initial value is such that $0< c \le G^{k+1/2}(\z) \le C$, the solution at time $t >0$ satisfies $0< c(t) \le G^{k+1/2}(\z) \le C(t)$, where $c(t)$ is a positive constant and $C(t)$ is a bounded constant.

The previous computations ensure, following the line of Theorem \ref{th:2}, that the unique nonnegative solution to the Fokker--Planck equation converge to the steady--state distribution \fer{eq:finfty_g}, without any explicit rate. However, as we will show in the next Section, numerical test in the two-dimensional case suggest that a polynomial rate of convergence could be found.

\section{Numerical tests}\label{sect:4}

In this Section we perform several numerical tests on the behavior of the solutions to the introduced models. First we check the consistency of the  Fokker-Planck equations with suitable microscopic particles' systems composed by a large number of particles, in the presence of interaction forces. Furthermore, we numerically investigate the 2D case where, as specified in the previous Section, rigorous results on the trends to equilibrium are not present. In all the tests the evolution of Fokker-Planck models have been obtained with the structure-preserving schemes defined  in \cite{PZ}, see also \cite{LZ} for related results. 

\subsection{Convergence of the particles' system}

In this test we design a particle system a system composed $N\gg 0$ particles evolving through stochastic differential equations (SDEs) describing the position $\x_i(t) \in \mathbb R^{d}$, $i = 1,\dots,N$ of the agents. In particular, we consider the system of equations given by 
\begin{equation}
\label{eq:part1}
\begin{split}
d\x_i(t) = &\left(\lambda(\x_i-\x_0) + \dfrac{\mu}{N}\sum_{j=1}^N P(\x_i,\x_j)(\x_j-\x_i) \right ) \mathds{1}_{D^c}(\x_i)dt \\
&+ \sqrt{2\sigma^2}d\mathbf W_i^t,\qquad \lambda,\mu>0
\end{split}
\end{equation}
with $\lambda + \mu = 1$ and where we denoted by $\left\{\mathbf W_i^t\right\}_{i=1}^N$ a vector of $N$ independent $d$-dimensional Wiener processes, $P(\cdot,\cdot) \in [0,1]$ and $\mathds{1}_{D^c}(\x_i)$ is the indicator function of the complement of $D \subset \mathbb R^d$. 

Since the drift of the particles' system is discontinuous, the transition to chaos of the particle system \eqref{eq:part1}  cannot be obtained through standard results based on drifts generated by globally Lipschitz interactions \cite{M} or based on convexity arguments \cite{BGM}. 

\subsubsection{Test 1a: Uniform interaction case}
In the following we compare the evolution of the reconstructed distribution $f^{N}_1(x,t)$ of the particles' system $\{x_i(t)\}_{i=1}^N$ solution to  \eqref{eq:part1} with the distribution solution of the Fokker-Planck model \eqref{eq:FP}. The reconstruction $f^{N}_1(x,t)$ has been obtained through a standard histogram. On the other hand, we will indicate with $f_1(\x,t)$ the solution at time $t\ge0$ of the model \eqref{eq:FP}. 

\begin{figure}
\centering
\includegraphics[scale = 0.30]{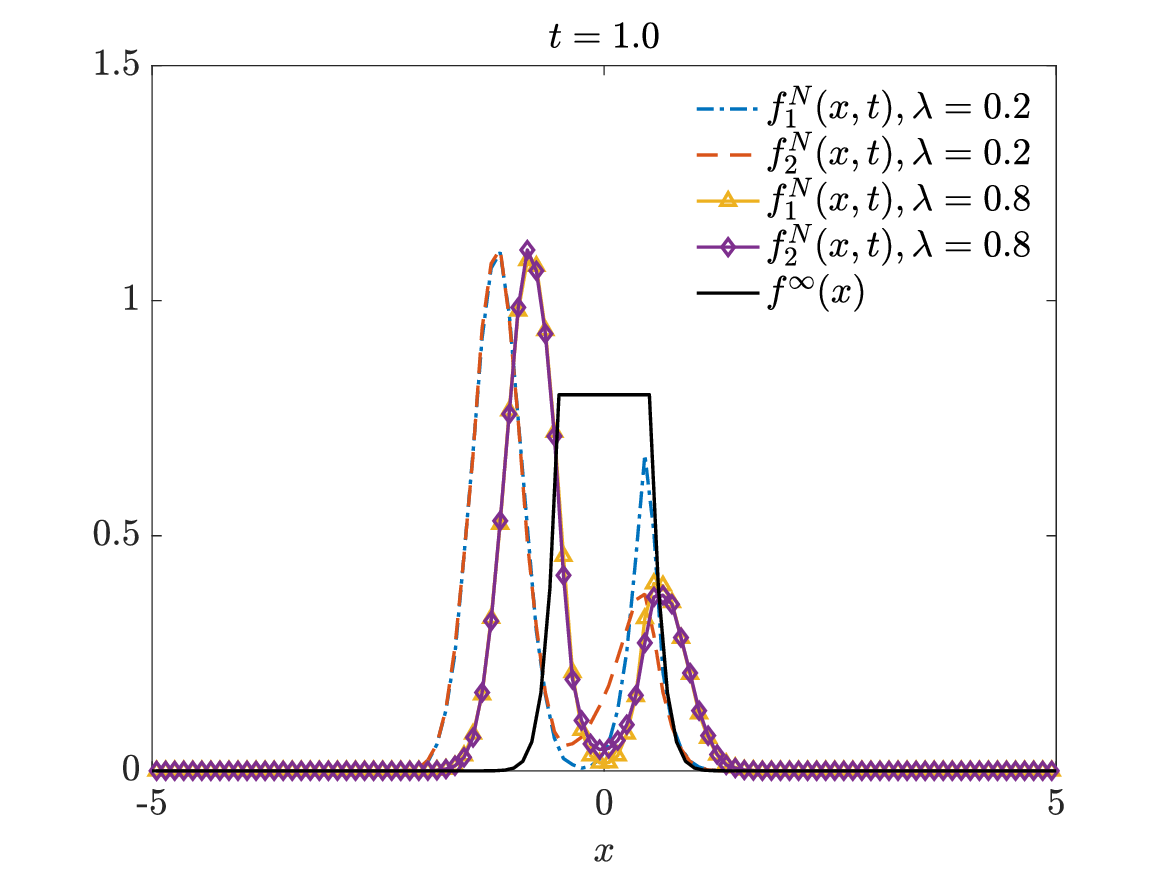}
\includegraphics[scale = 0.30]{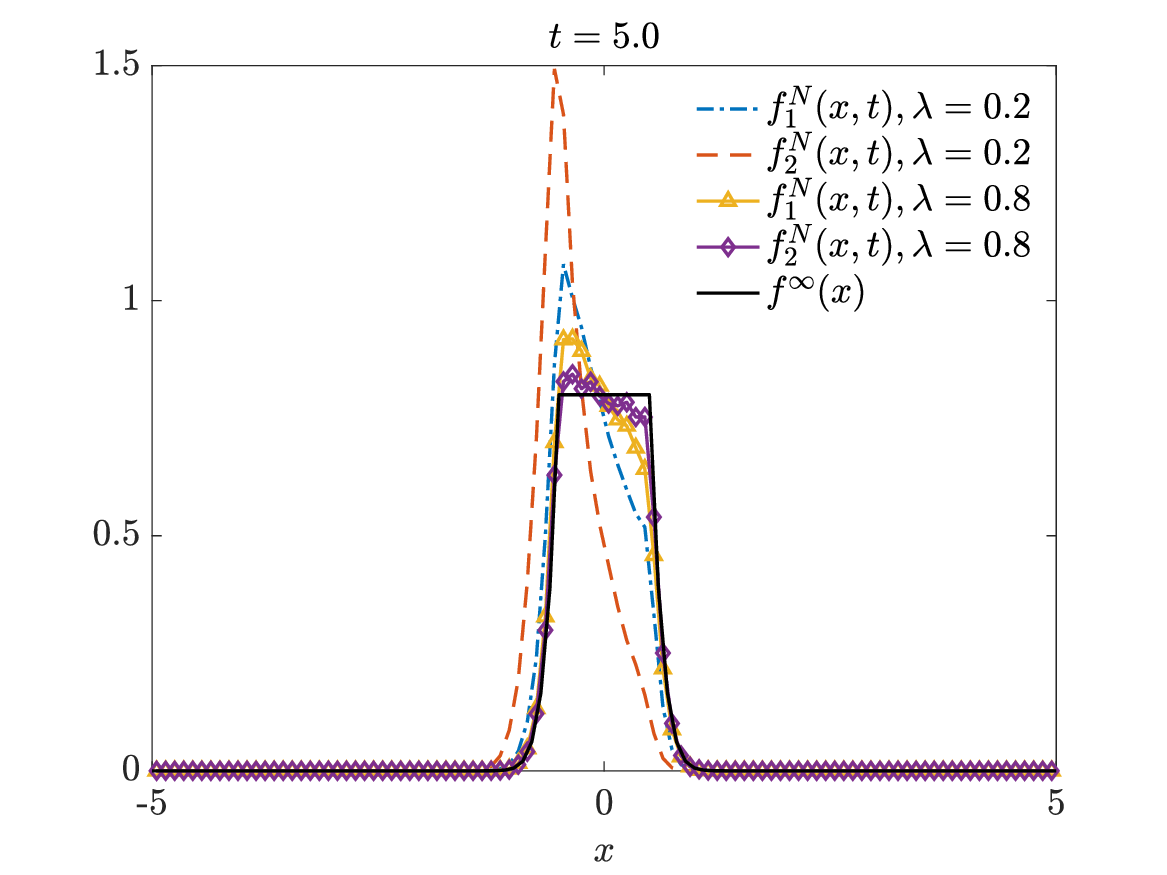}
\includegraphics[scale = 0.30]{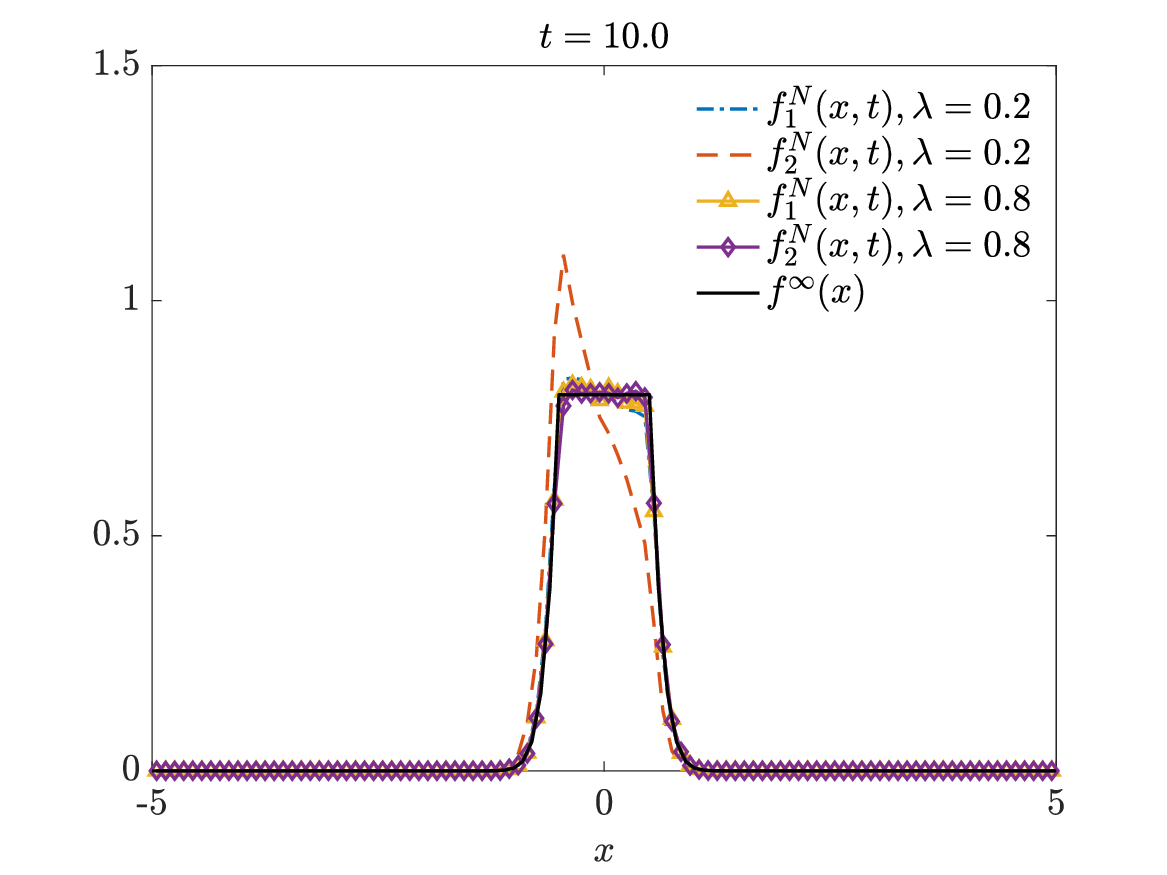}
\includegraphics[scale = 0.30]{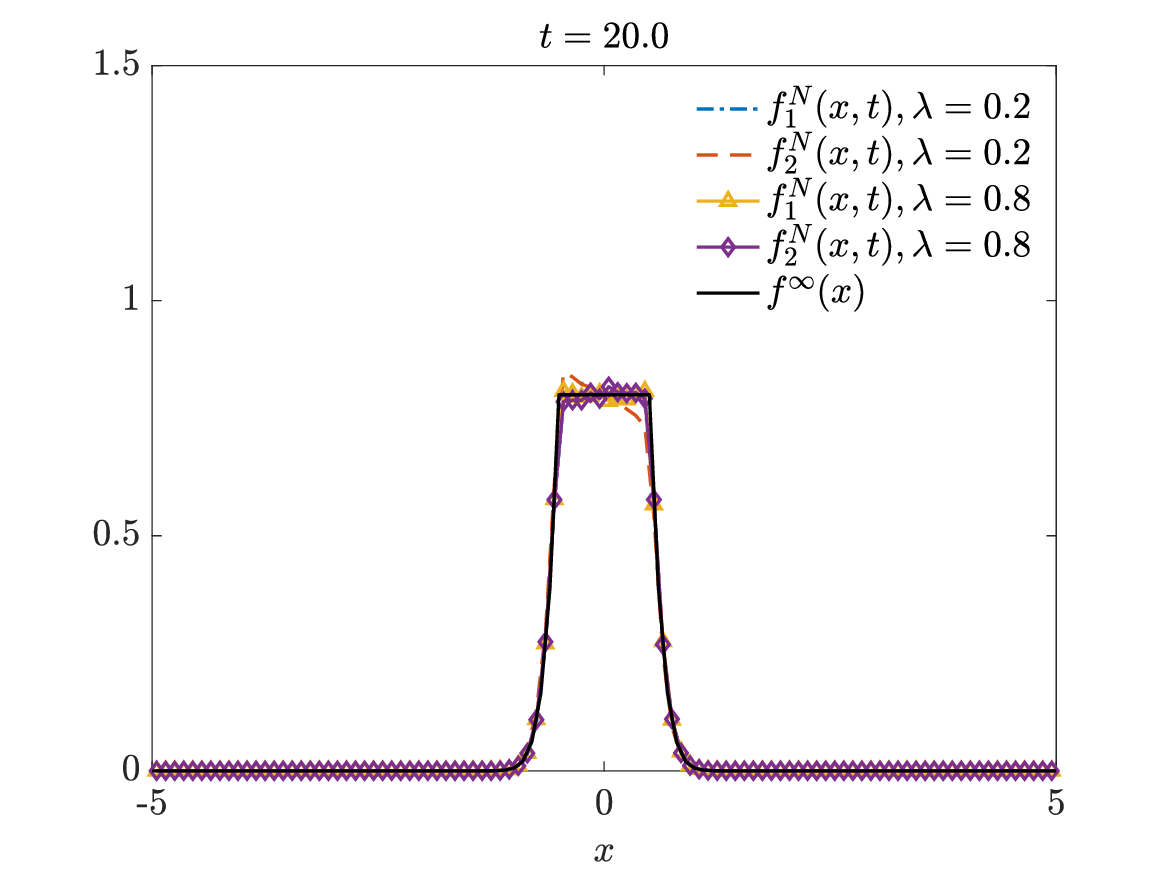}
\caption{\textbf{Test 1a}. Evolution of the reconstructed distribution functions of the particles' systems \eqref{eq:part1}-\eqref{eq:part2} given by $f^{N}_1(x)$, $f^{N}_2(x)$ in the case $P\equiv1$ at different times $t = 1,5,10,20$ and for $\lambda = 0.2,0.8$. We considered an Euler-Maruyama scheme with $N = 10^5$ and $\Delta t = 10^{-2}$, the histograms have been obtained in the interval $[-5,5]$ with $N_x = 101$ gridpoints and the target domain is $D = \{x \in \mathbb R: |x-x_0|\le \frac{1}{2}\}$, $x_0=0$ and $m_1,\sigma^2>0$  solution to the system \eqref{eq:system1} in such a way $m_2 = 0.8$ and $\delta = \frac{1}{2}$. }
\label{fig:lambda_infty}
\end{figure}

In the case $P\equiv 1$, we established in Section  \ref{sect:22} the equivalence  in terms of large-time asymptotics between \eqref{eq:FP} and the surrogate model with continuous drift \eqref{eq:FP_P1}. We will indicate with $f_2(\x,t)$ the solution at time $t\ge 0$ of the model \eqref{eq:FP_P1}. It is worth to remark that the transient regime of \eqref{eq:FP} and \eqref{eq:FP_P1} are not equal even though they both converge towards the same equilibrium state $f^\infty(\x)$. The surrogate model can be obtained from the following particles' system
\begin{equation}
\label{eq:part2}
\begin{split}
d\x_i(t) = &\left(\x_i - \tilde{\x}_i \right )dt + \sqrt{2\kappa(\x_i,t)}d\mathbf W_i^t,
\end{split}
\end{equation}
where, as before, we denoted by $\left\{\mathbf W_i^t\right\}_{i=1}^N$ a vector of $N$ independent $d$-dimensional Wiener processes, $\tilde{\x}_i = \lambda \x_0 + \mu \bar{\mathbf{u}}(t)$, being $\bar{\mathbf{u}}(t) = \frac{1}{N} \sum_{i=1}^N \x_i(t)$, and $\kappa(\cdot,\cdot)$ has been defined in \eqref{eq:FP_simp_diffusion}. Hence, we indicate with $f_2^{N}(\x)$ the reconstructed distribution of the particles' system $\{x_i(t)\}_{i=1}^N$ solution to \eqref{eq:part2}. 

In Figure \ref{fig:lambda_infty} we study the convergence of $f_1^{N}(x)$ and $f_2^{N}(x)$ for large times towards $f^\infty(x)$ defined in \eqref{eq:finfty_1} in the one dimensional case $d = 1$. We considered the values of $\lambda = 0.2,0.8$ and we report the empirical densities at times^^>$t=1,5,10,20$. 

As initial distribution we considered a sum of Gaussian densities centered in $x = \pm 2$ of the form
\begin{equation}
\label{eq:f0_test1}
f(x,0) = \dfrac{3}{4\sqrt{2\pi \sigma_0^2}}\exp\left\{-\dfrac{(x+2)^2}{2\sigma_0^2}\right\} + \dfrac{1}{4\sqrt{2\pi \sigma_0^2}}\exp\left\{-\dfrac{(x-2)^2}{2\sigma_0^2}\right\}, 
\end{equation}
with $\sigma^2_0 = \frac{1}{10}$.
Furthermore, we fixed the target domain $D = \{x \in \mathbb R: |x-x_0|\le \frac{1}{2}\}$ with $x_0 = 0$. The diffusion coefficient $\sigma^2$ have been obtained to guarantee that $m_2 = 0.8$, i.e. with probability $0.8$ a particle lies inside $D$ for long times.  Hence, we constructed an initial sample of particles $\{x_i(0)\}_{i=1}^N$ whose distribution is $f(x,0)$. The evolution of the particles' systems \eqref{eq:part1}-\eqref{eq:part2} have been obtained through an Euler-Maruyama scheme with $\Delta t = 10^{-2}$ on a set of $N = 10^5$ particles. We may observe how both $f_1^{N}(x,t)$ and $f_2^{N}(x,t)$ converge in time to the analytical $f^\infty(x)$ as discussed in Section \ref{sect:2}. Furthermore, we may observe how the dynamics described by $f_1^{N}(x,t)$ and $f_2^{N}(x,t)$ are different for all $t\ge0$ finite.  

\begin{figure}
\centering
\includegraphics[scale = 0.2]{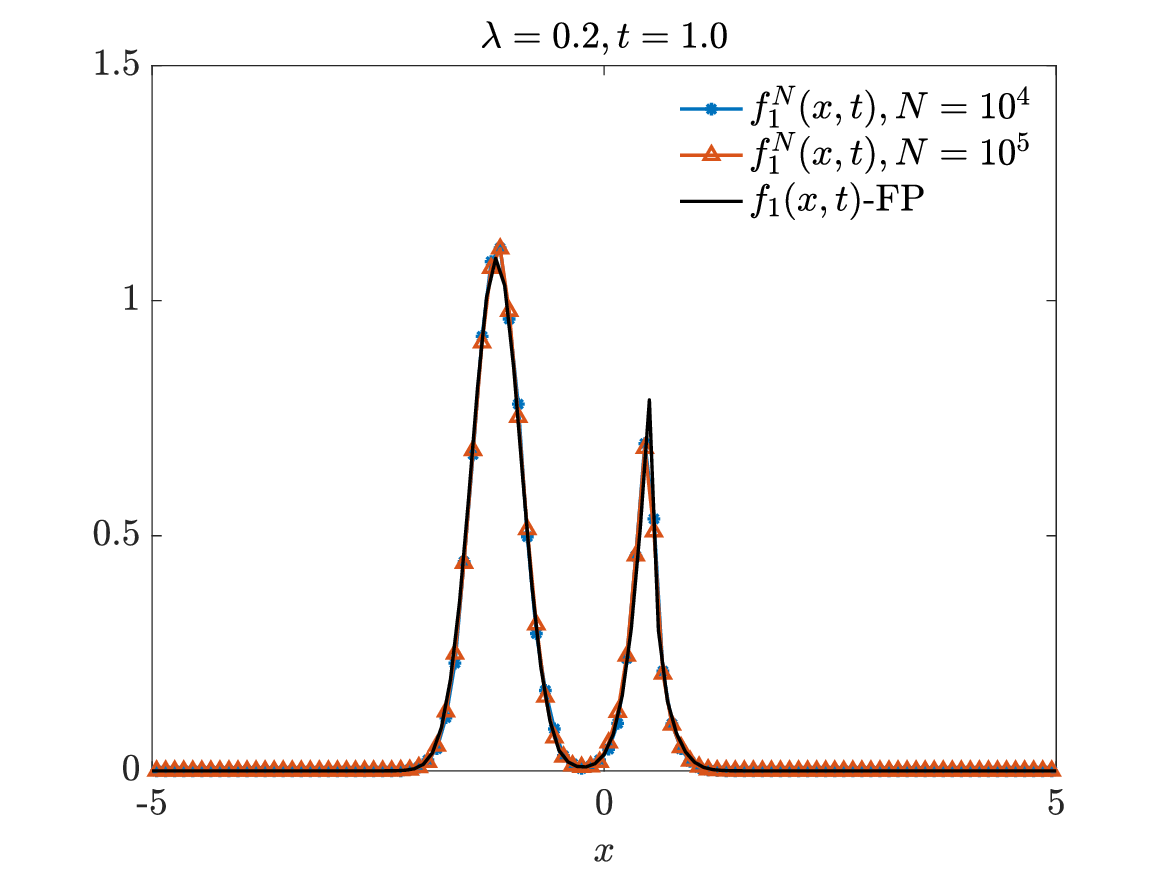}
\includegraphics[scale = 0.2]{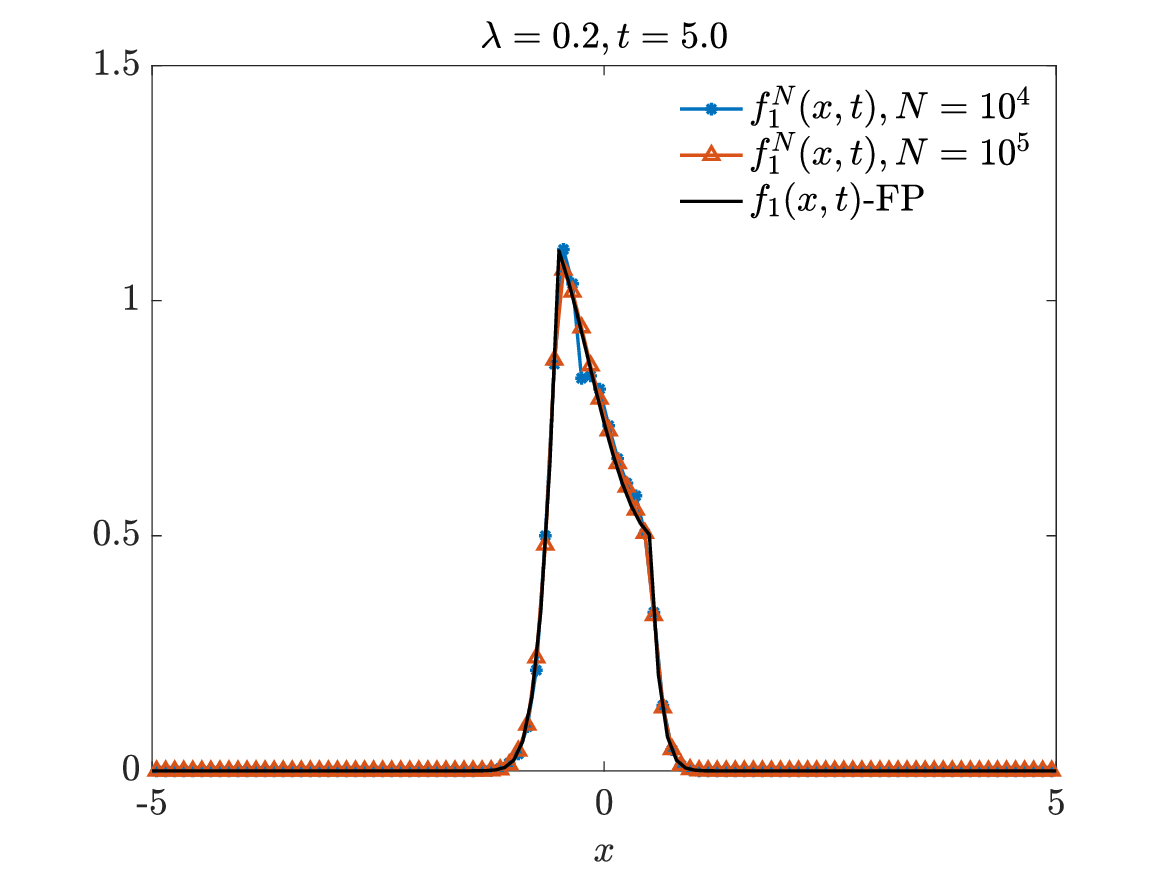}
\includegraphics[scale = 0.2]{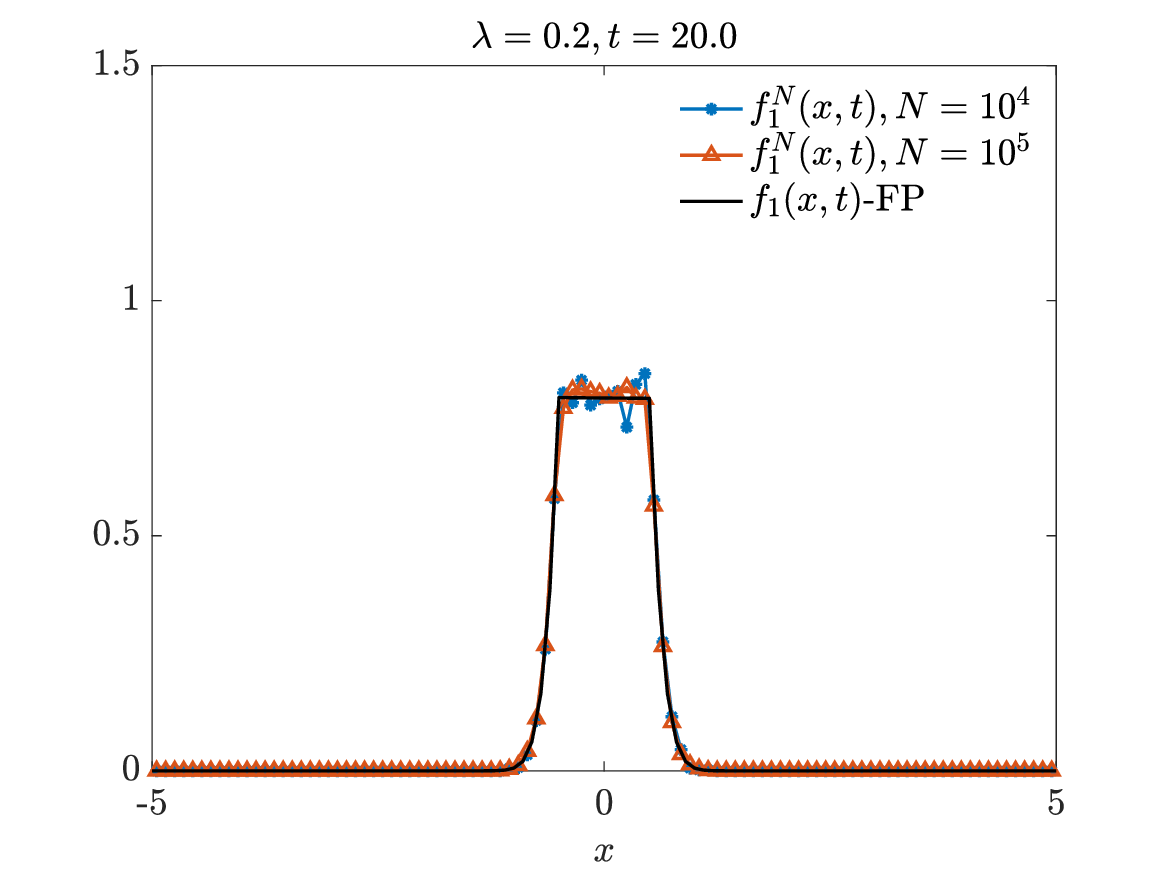}\\
\includegraphics[scale = 0.2]{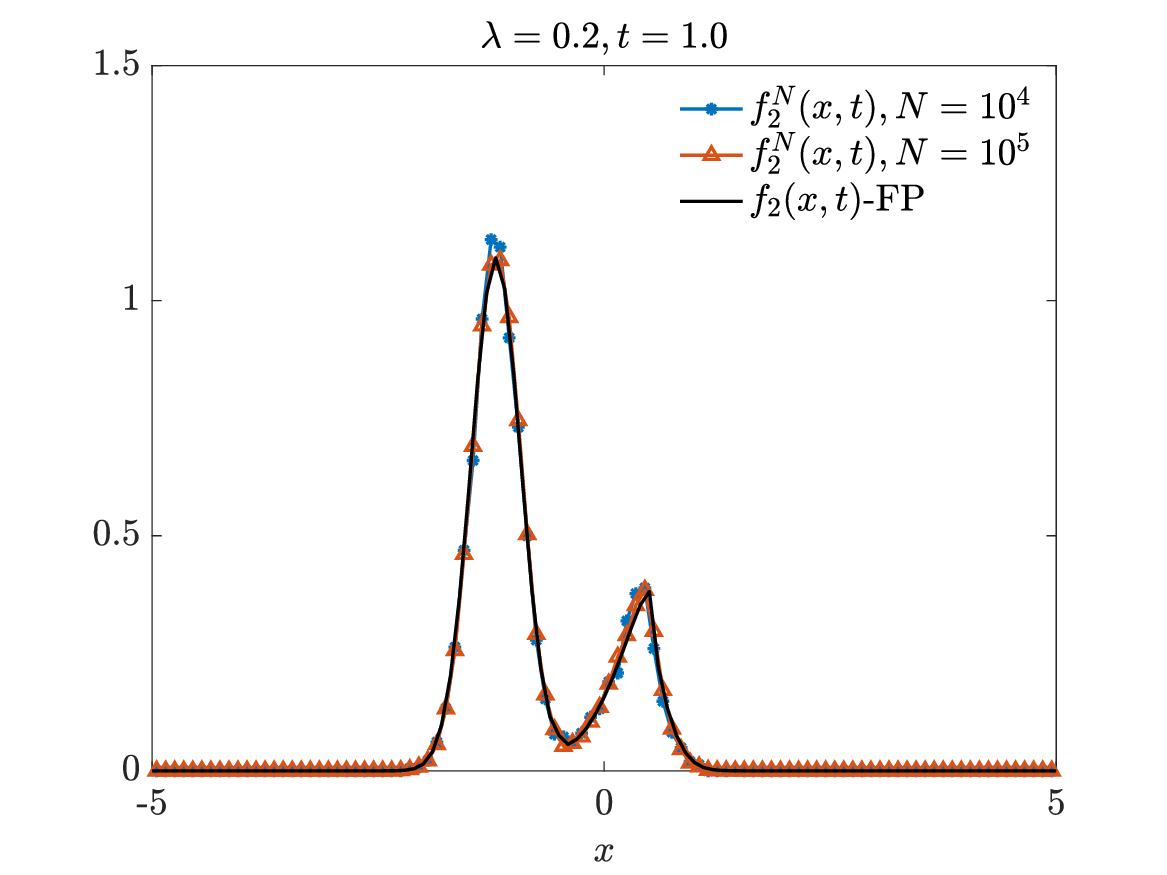}
\includegraphics[scale = 0.2]{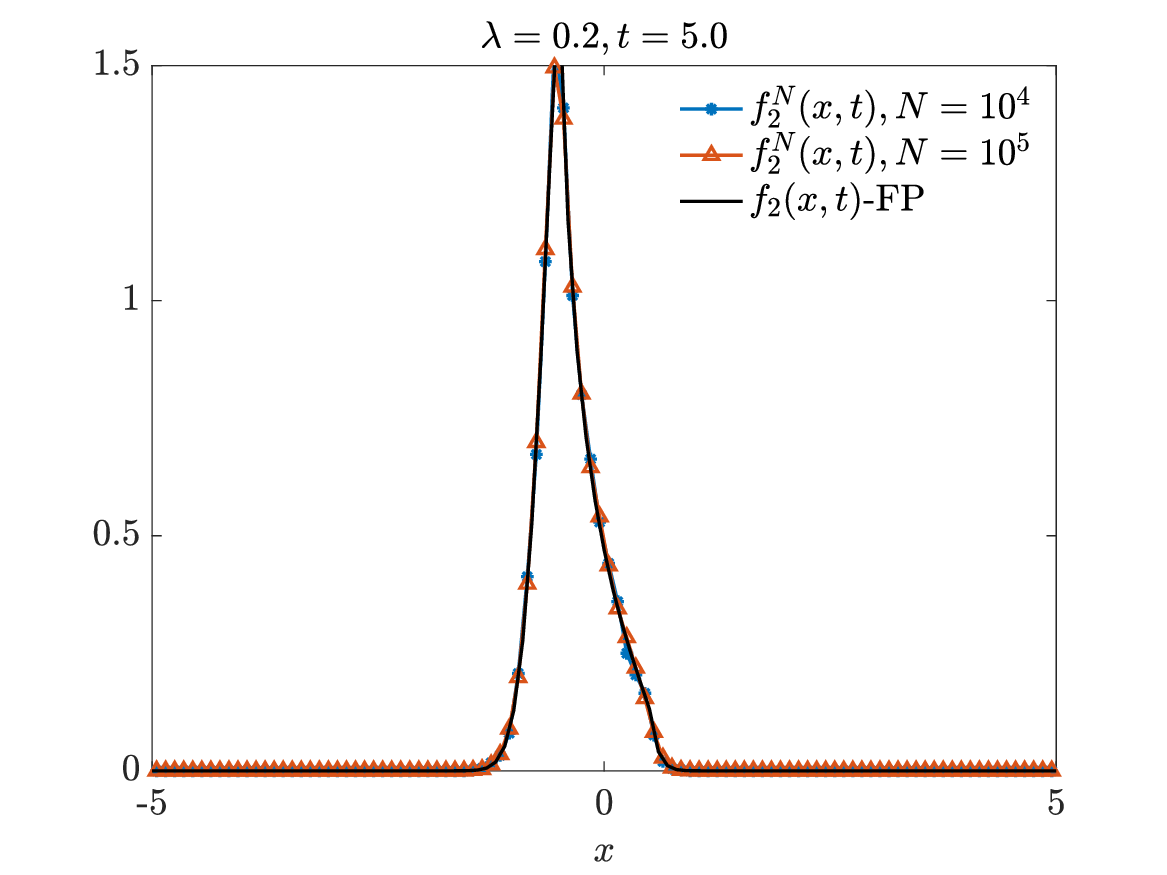}
\includegraphics[scale = 0.2]{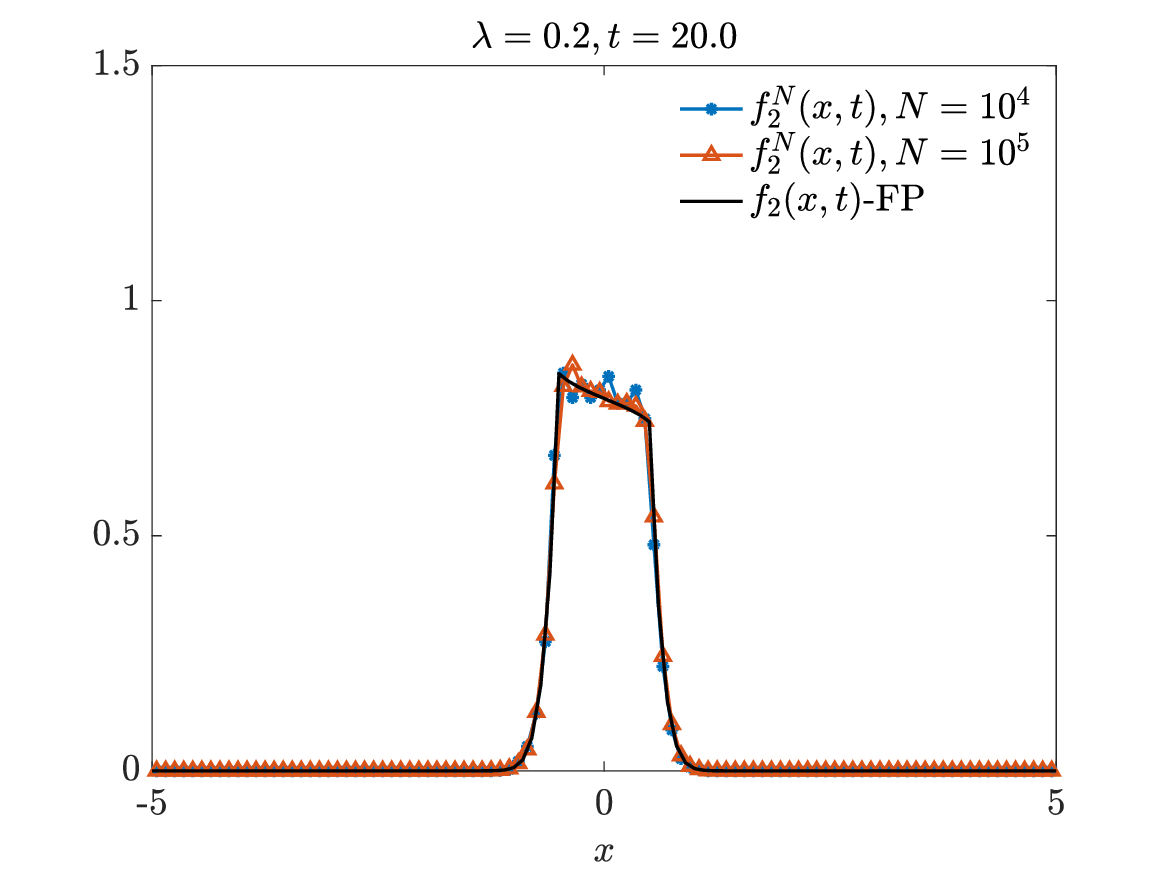}
\caption{\textbf{Test 1a}. Comparison between the reconstructed distribution functions $f_1^{N}(x,t)$ (top row) and $f_2^{N}(x,t)$ (bottom row) for increasing $N = 10^{4}$, $N = 10^5$ of the particles' systems \eqref{eq:part1}-\eqref{eq:part2} in the case $P\equiv 1$, with the numerical solution of Fokker-Planck models defined in \eqref{eq:FP}-\eqref{eq:FP_P1} and denoted by $f_1(x,t)$, $f_2(x,t)$.  We considered $\lambda = 0.2$,  a discretization of the interval $[-5,5]$ obtained with $N_x = 101$ gridpoints. The target domain is $D = \{x \in \mathbb R: |x-x_0|\le \frac{1}{2}\}$, $x_0=0$ and $m_1,\sigma^2>0$ solution to \eqref{eq:system1} with $m_2 = 0.8$ and $\delta = \frac{1}{2}$. Initial condition defined in \eqref{eq:f0_test1}.}
\label{fig:f12_N}
\end{figure}


Hence, to test the consistency of the particles' systems \eqref{eq:part1}-\eqref{eq:part2} we compare the evolutions of $f_1^{N}(x,t)$ and $f_2^{N}(x,t)$ for increasing $N\gg0$ with the respective solutions to the Fokker-Planck models \eqref{eq:FP} and \eqref{eq:FP_P1} defined as $f_1(x,t)$ and $f_2(x,t)$. In Figure \ref{fig:f12_N} we test numerically the consistency of the Fokker-Planck descriptions with the densities $f_1^{N}$ (top row) and $f_2^N$ (bottom row) characterizing the particles systems \eqref{eq:part1}-\eqref{eq:part2} respectively and obtained with $N = 10^4, 10^5$. The reconstruction of the densities have been done over the interval $[-5,5]$ discretized with $N_x = 101$ gridpoints. We plotted with the black line the numerical approximation of the Fokker-Planck models over $[-5,5]$ discretized with $N_x = 101$ gridpoints. It is observed how, for an increasing number of particles the Fokker-Planck models become consistent with the particle dynamics. 

\begin{figure}
\centering
\includegraphics[scale = 0.30]{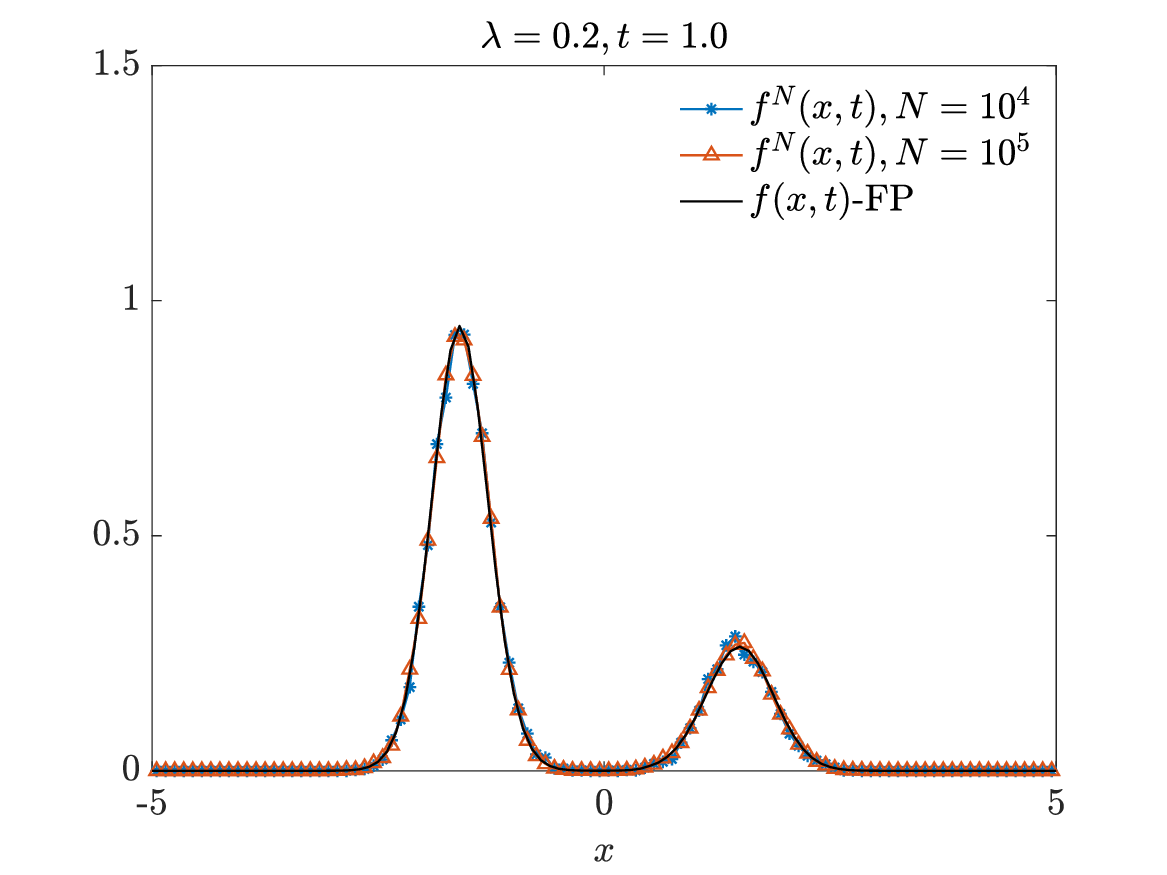}
\includegraphics[scale = 0.30]{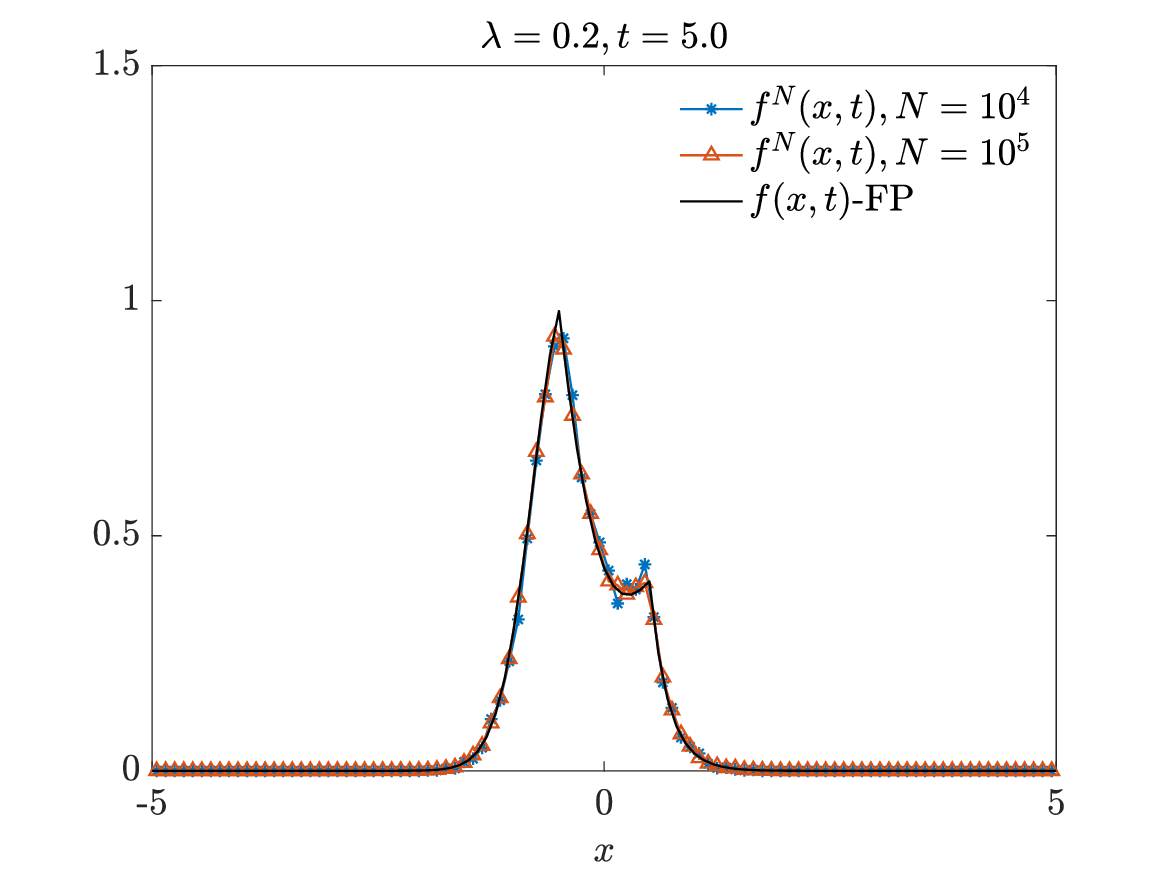}
\includegraphics[scale = 0.30]{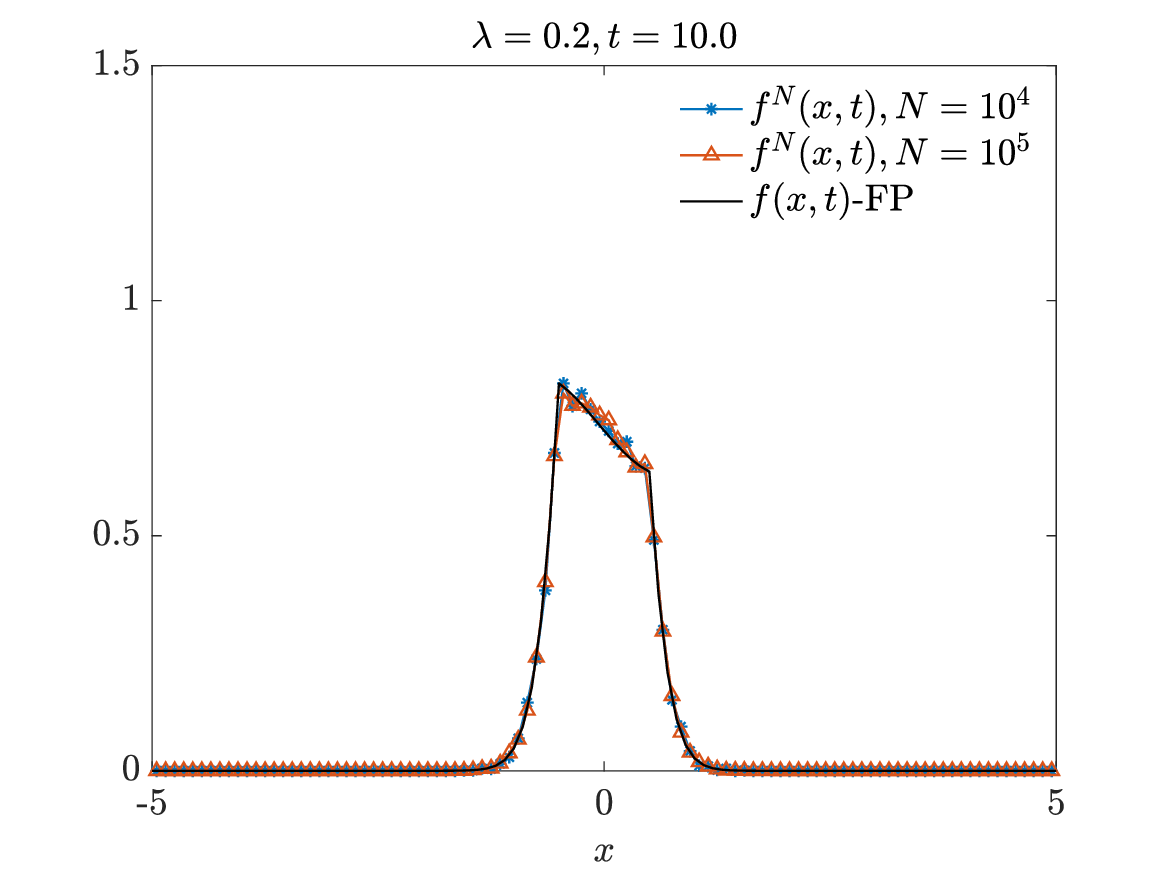}
\includegraphics[scale = 0.30]{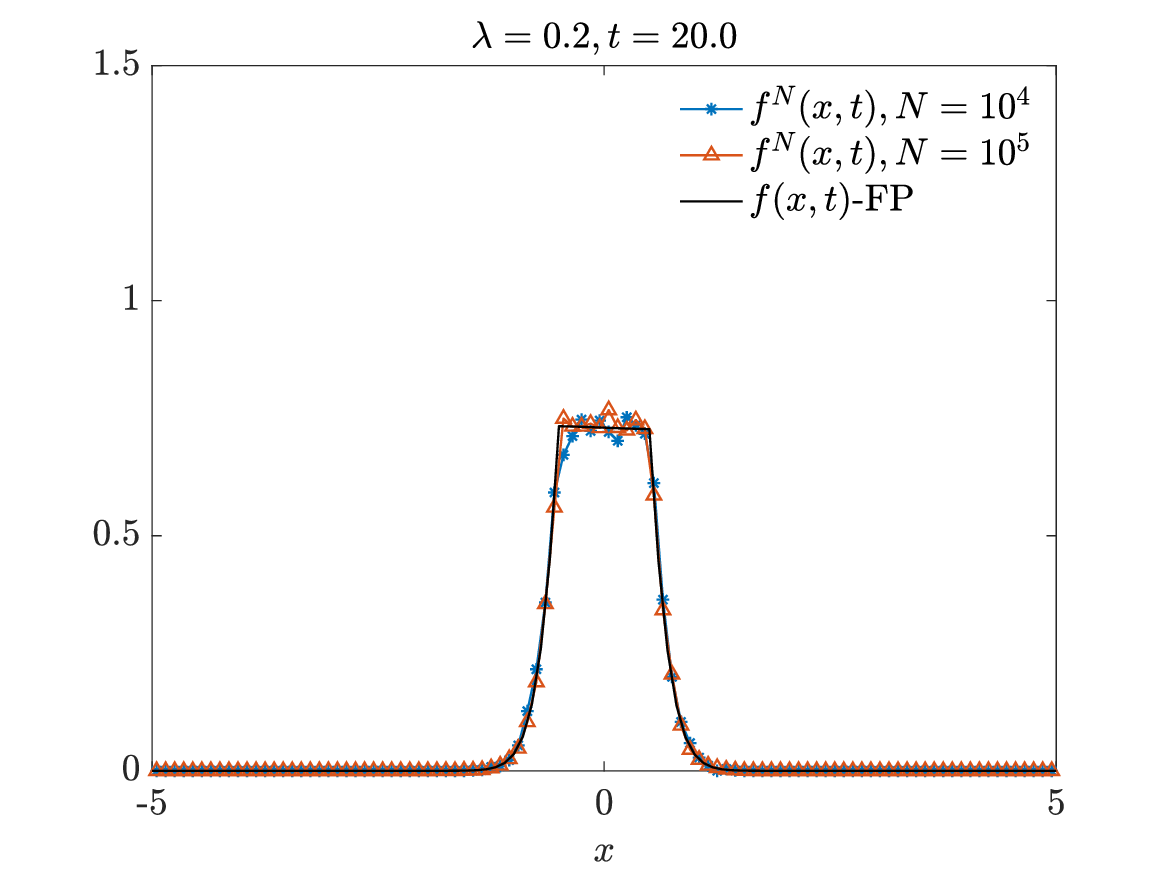}
\caption{\textbf{Test 1b}. Comparison between the reconstructed distribution function $f^{N}(x,t)$  for increasing $N = 10^{4}$, $N = 10^5$ of the particles' systems \eqref{eq:part1} in the case $P(x,y)$ in \eqref{eq:CS}, with the numerical solution of Fokker-Planck models defined in \eqref{eq:FP} and denoted by $f(x,t)$.  We considered $\lambda = 0.2$,  a discretization of the interval $[-5,5]$ obtained with $N_x = 101$ gridpoints. The target domain is $D = \{x \in \mathbb R: |x-x_0|\le \frac{1}{2}\}$, $x_0=0$ and $\sigma^2 = 0.2$ of Figure \ref{fig:f12_N} and $\delta = \frac{1}{2}$. Initial condition defined in \eqref{eq:f0_test1}.}
\label{fig:CS1}
\end{figure}

\subsubsection{Test 1b: Nonuniform interaction forces }
In this test we consider a particles' system evolving through nonuniform interaction forces by considering the space-dependent  interaction function of the form
\begin{equation}
\label{eq:CS}
P(\x,\mathbf y) = \dfrac{1}{\left(1+|\x-\mathbf y|^2\right)^\gamma}, \qquad \gamma>0, \qquad \x,\mathbf{y} \in \mathbb R^d.
\end{equation}
in \eqref{eq:FP}. In all the following  tests we will fix $\gamma = 1$. It is worth to remark that the form of $P(\cdot,\cdot)>0$ is consistent with well known Cucker-Smale-type models for swarming of large flocks \cite{CS}, see also \cite{CFRT,CFTV} for a review of connected kinetic models with nonlocal interactions. In Figure \ref{fig:CS1} we present the evolution of the numerical approximation of the Fokker-Planck model \eqref{eq:FP}, with $d = 1$ and interactions defined in \eqref{eq:CS} and $\lambda = 0.2$, compared with the particles' dynamics \eqref{eq:part1} in the one dimensional case. As initial condition as considered \eqref{eq:f0_test1} and the target domain is  $D = \{x \in \mathbb R:|x-x_0|\le \frac{1}{2}\}$ with $x_0=0$. We considered an increasing number of particles $N = 10^4$ and $N = 10^5$ whose dynamics is integrated in the time interval $[0,20]$, $\Delta t = 10^{-2}$. For all times, the evolution of the particles distribution $f^{N}(x,t)$ is consistent with the solution to the numerical Fokker-Planck model \eqref{eq:FP}. 


In order to compare the influence of interactions on the convergence to the target domain we report in Figure \ref{fig:uevo} the evolution of the mean position of the swarm both the cases \eqref{eq:part1}-\eqref{eq:part2} with $P\equiv1 $ and in the case defined by \eqref{eq:part1} with position dependent interaction given by $P(\cdot,\cdot)$ as in \eqref{eq:CS}. We will denote with $\bar u_1$ the mean position of the dynamics  \eqref{eq:part1} with $P\equiv 1$, with $\bar u_2$ the mean position of the dynamics \eqref{eq:part2} and, finally, with $\bar u(t)$ the mean position of the dynamics \eqref{eq:part1} with Cucker-Smale interactions \eqref{eq:CS}. We considered both a dynamics where the information on $x_0 \in D$ domain is characterized by $\lambda = 0.2$ (left) and the case $\lambda = 0.8$ (right). It is easily observed how for a high value of $\lambda \in [0,1]$ the three dynamics reach fast the target. On the other hand, for a small $\lambda \in [0,1]$, we can observe that in the case of uniform interactions the mean position of the swarm reaches faster the target for dynamics of the type \eqref{eq:part1} whereas the case of nonconstant diffusion defined in \eqref{eq:part2} is slower in reaching the target. Furthermore, for the considered initial distribution of particles, nonuniform space-dependent interactions are faster than uniform interactions with nonconstant diffusions. 


\begin{figure}
\centering
\includegraphics[scale = 0.3]{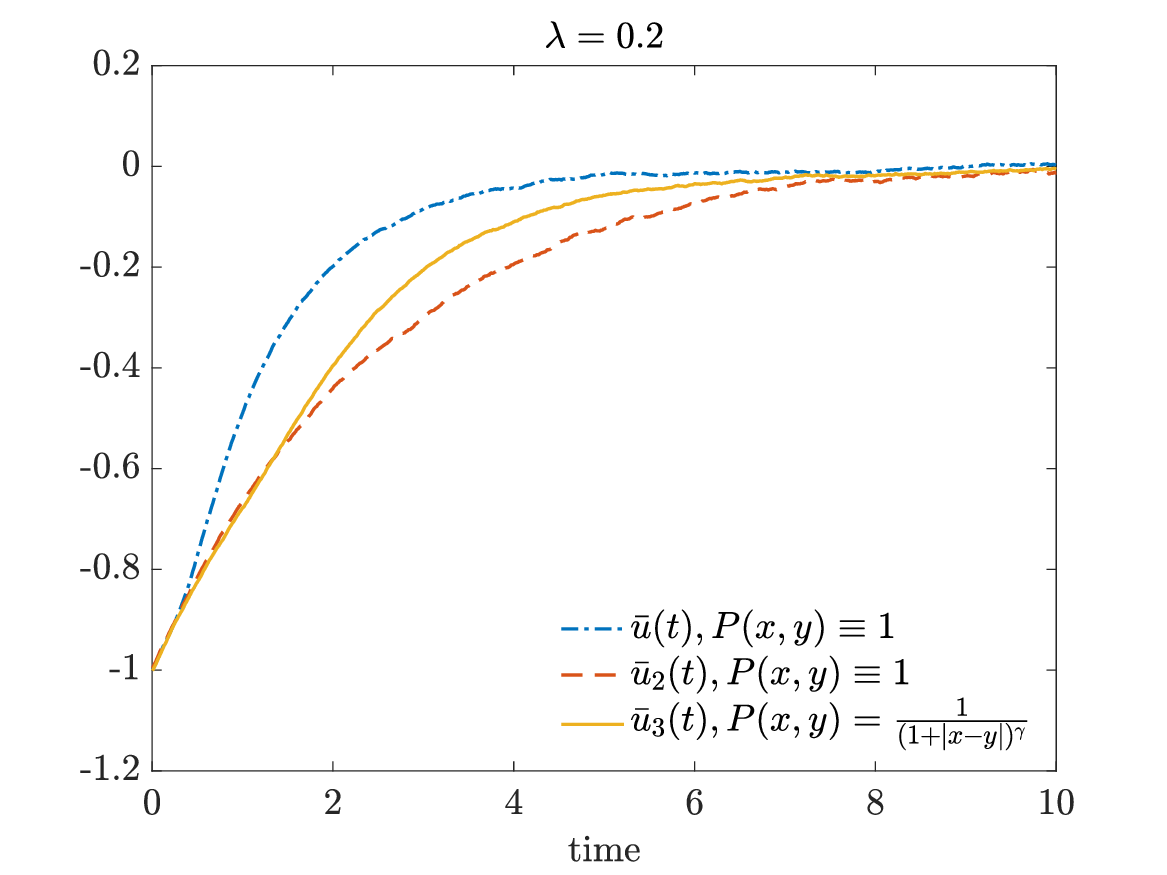}
\includegraphics[scale = 0.3]{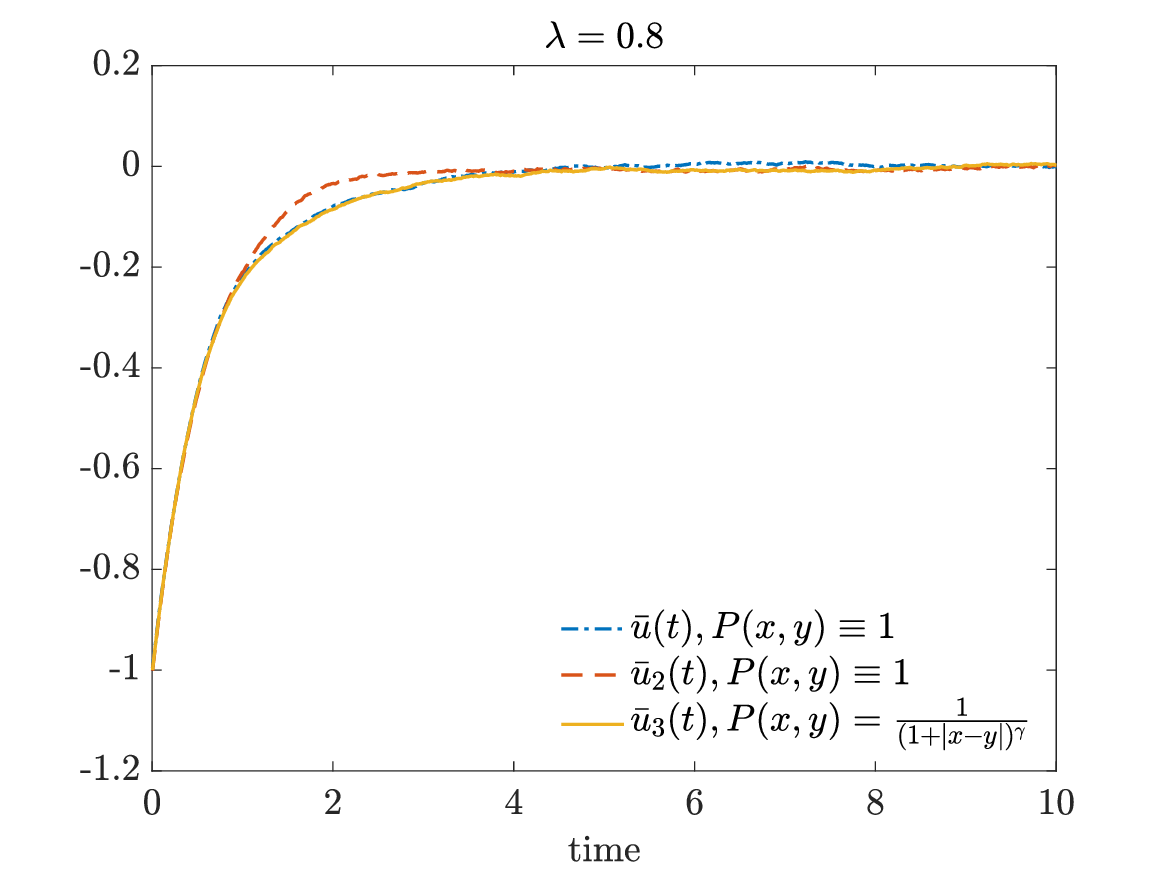}
\caption{\textbf{Test 1b}. Evolution of the particles' mean position. We denote with $\bar u_1(t)$ the mean position obtained from \eqref{eq:part1}, with $\bar u_2(t)$ the one obtained from \eqref{eq:part2} and with $\bar u_3(t)$ the one obtained with space dependent interactions $P(x,y)$ \eqref{eq:CS} in \eqref{eq:part1}. In all the tests we considered $N = 10^4$ particles evolving over the time interval $[0,10]$ with $\Delta t = 10^{-2}$. The target domain is $D = \{x \in \mathbb R: |x-x_0|\le \frac{1}{2}\}$ and we fixed the relevant parameters of Figure \ref{fig:CS1}. We considered the case $\lambda = 0.2$ (left) and $\lambda = 0.8$ (right). Initial distribution defined in \eqref{eq:f0_test1}. }
\label{fig:uevo}
\end{figure}

\subsubsection{Test 1c: 2D case}

\begin{figure}
\centering
\includegraphics[scale = 0.20]{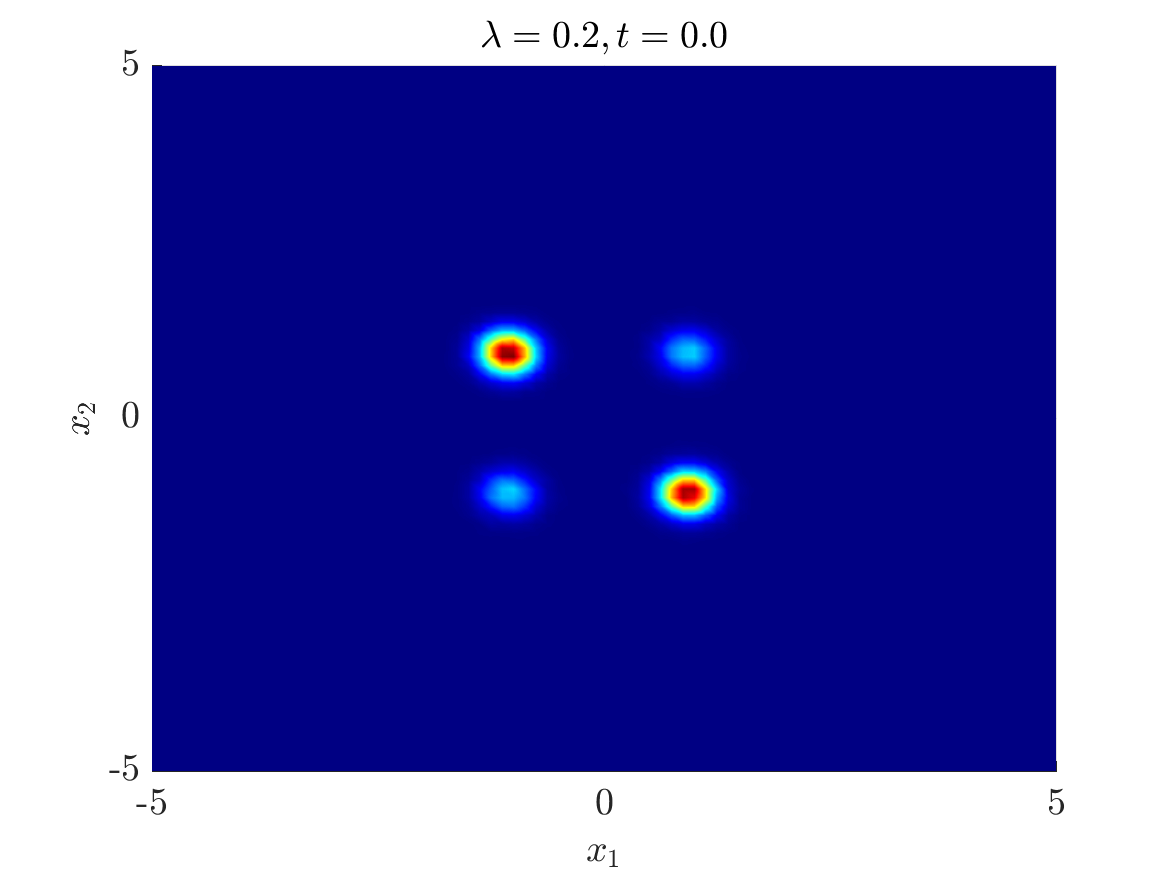}
\includegraphics[scale = 0.20]{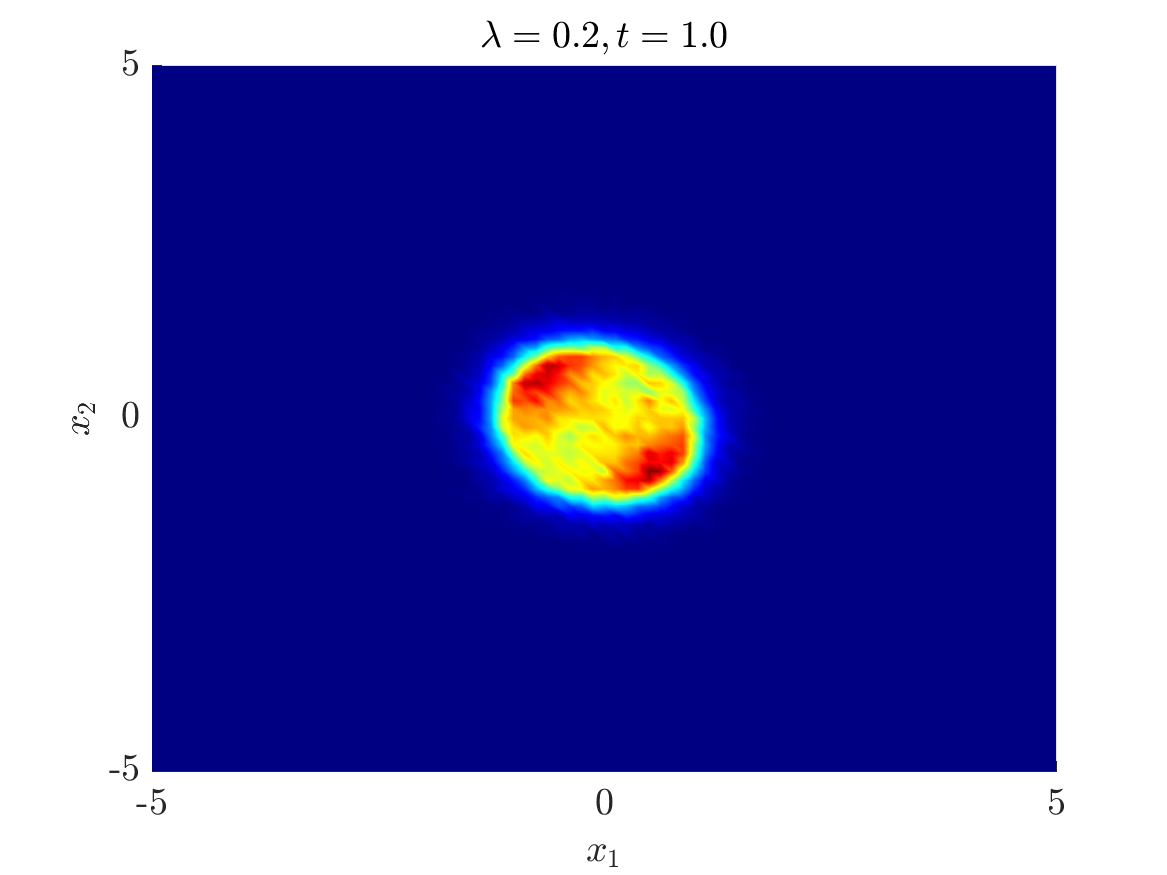}
\includegraphics[scale = 0.20]{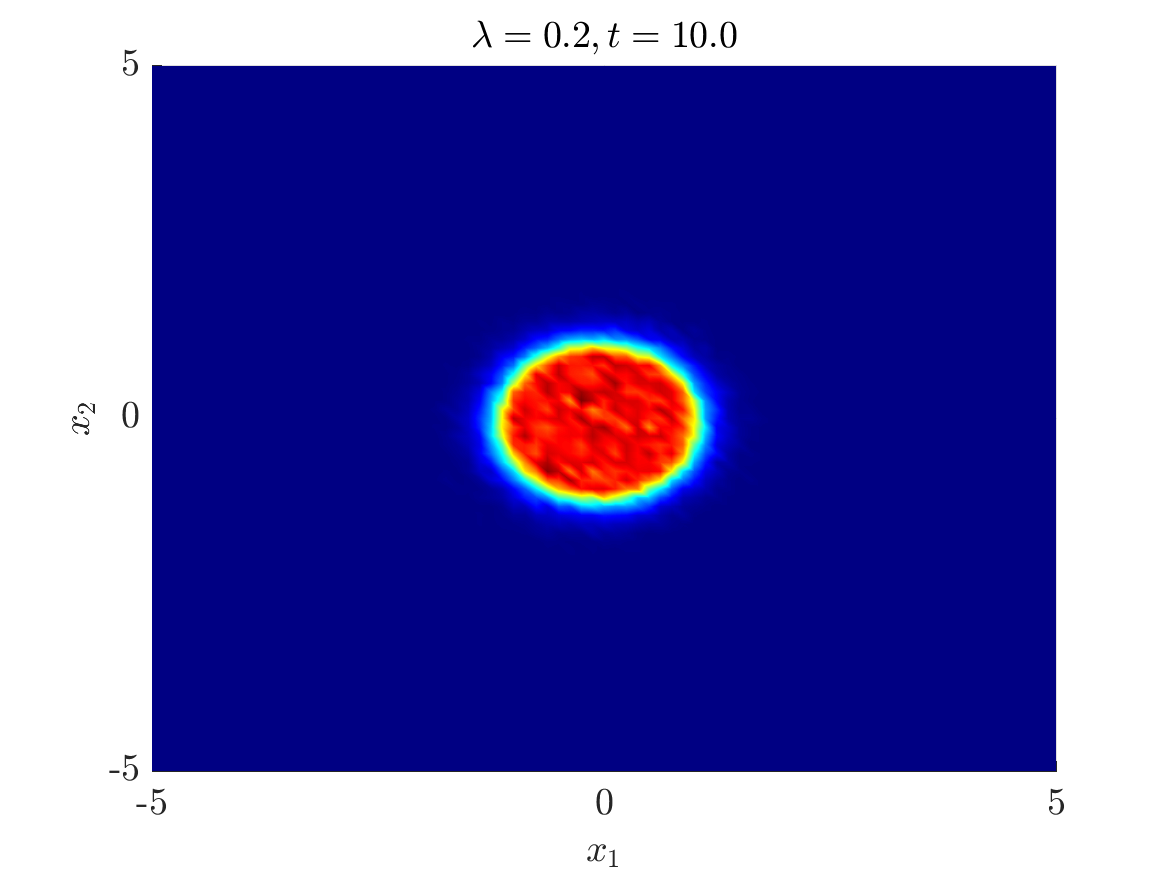}\\
\includegraphics[scale = 0.2]{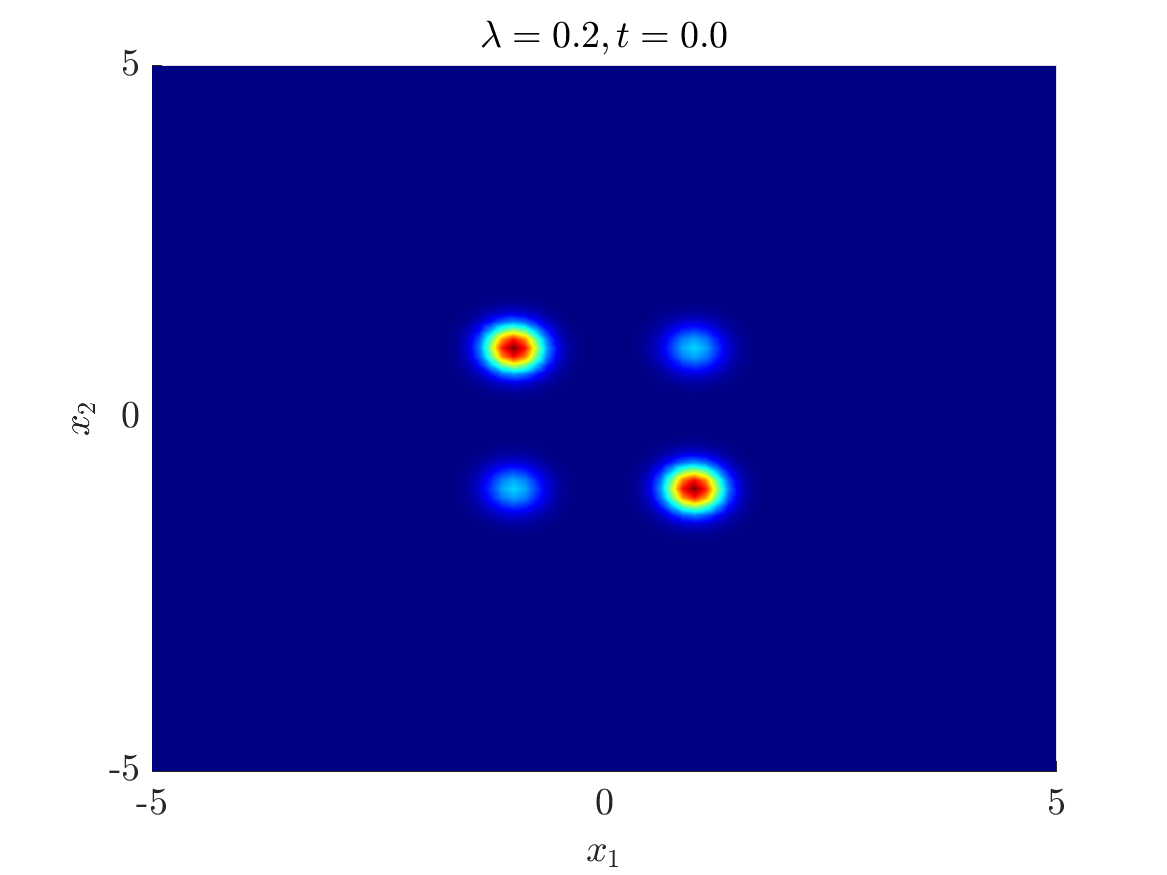}
\includegraphics[scale = 0.2]{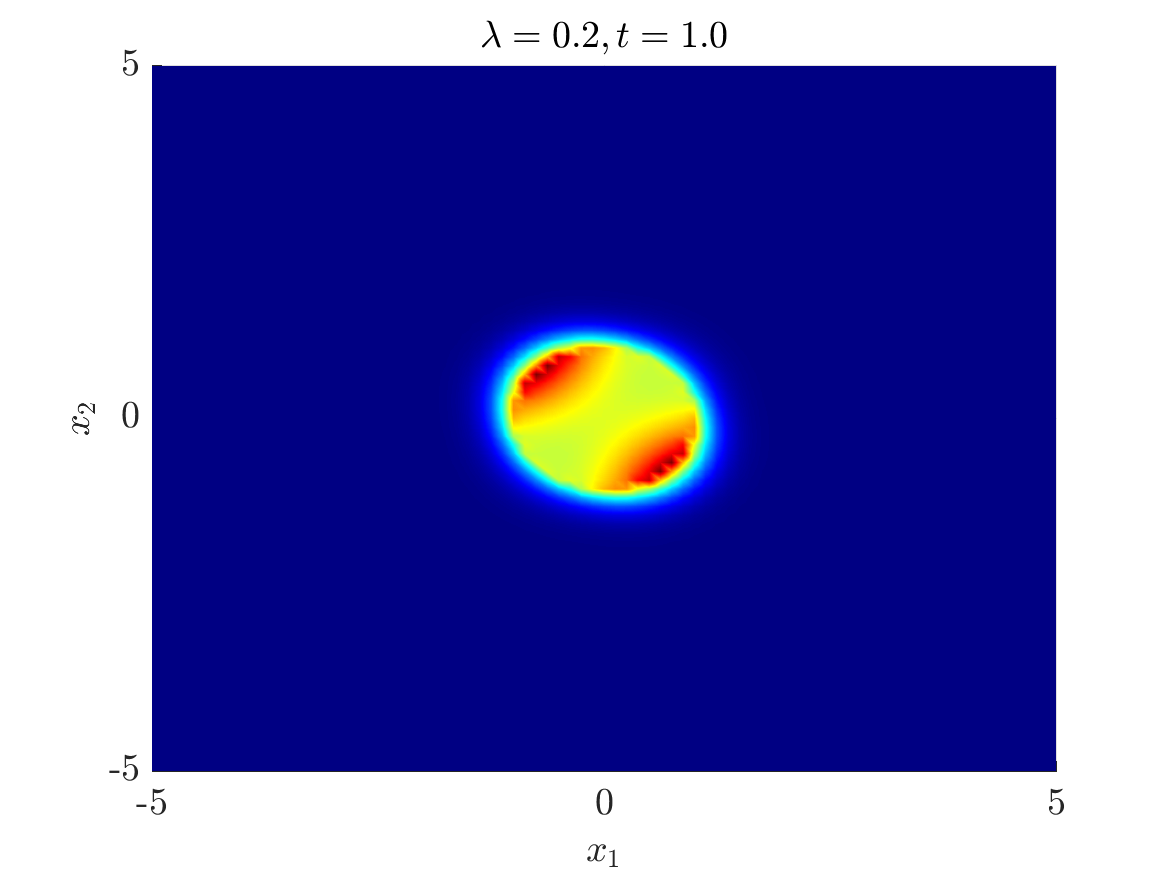}
\includegraphics[scale = 0.2]{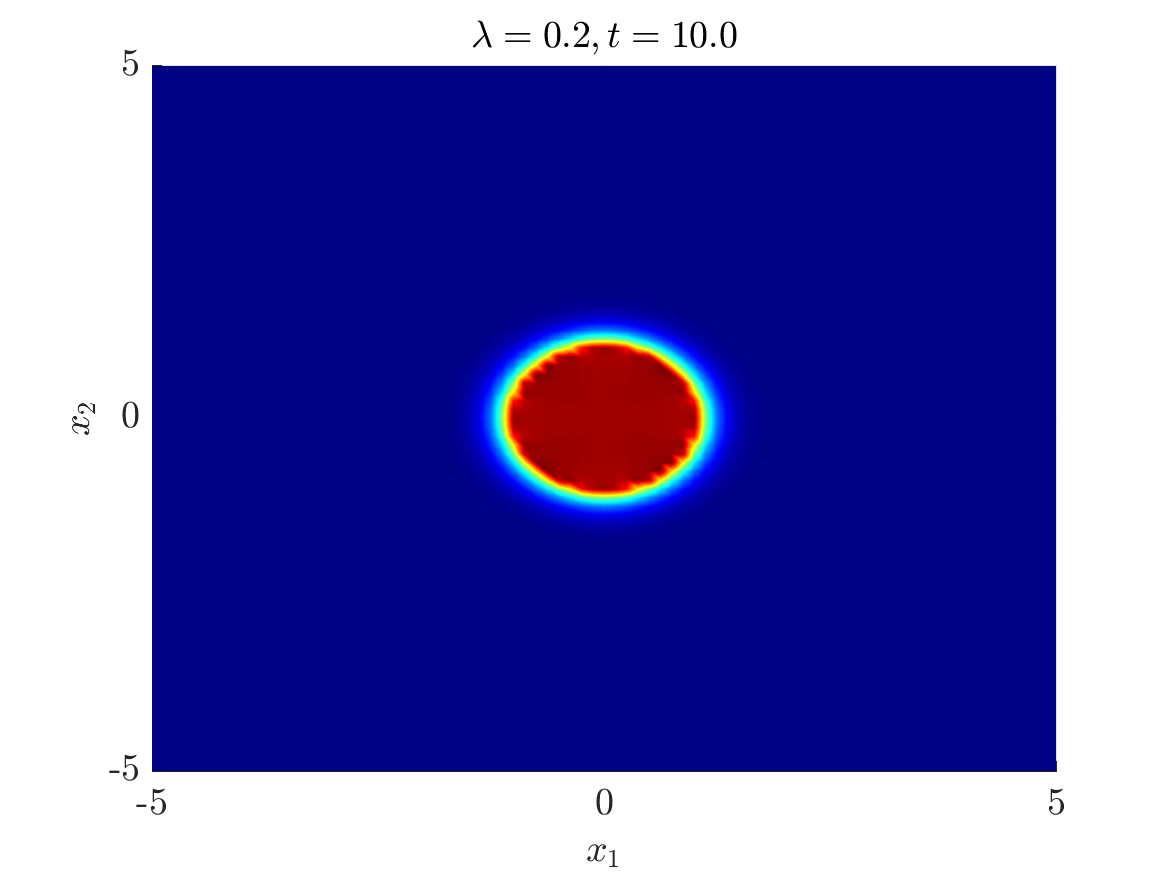}\\
\includegraphics[scale = 0.2]{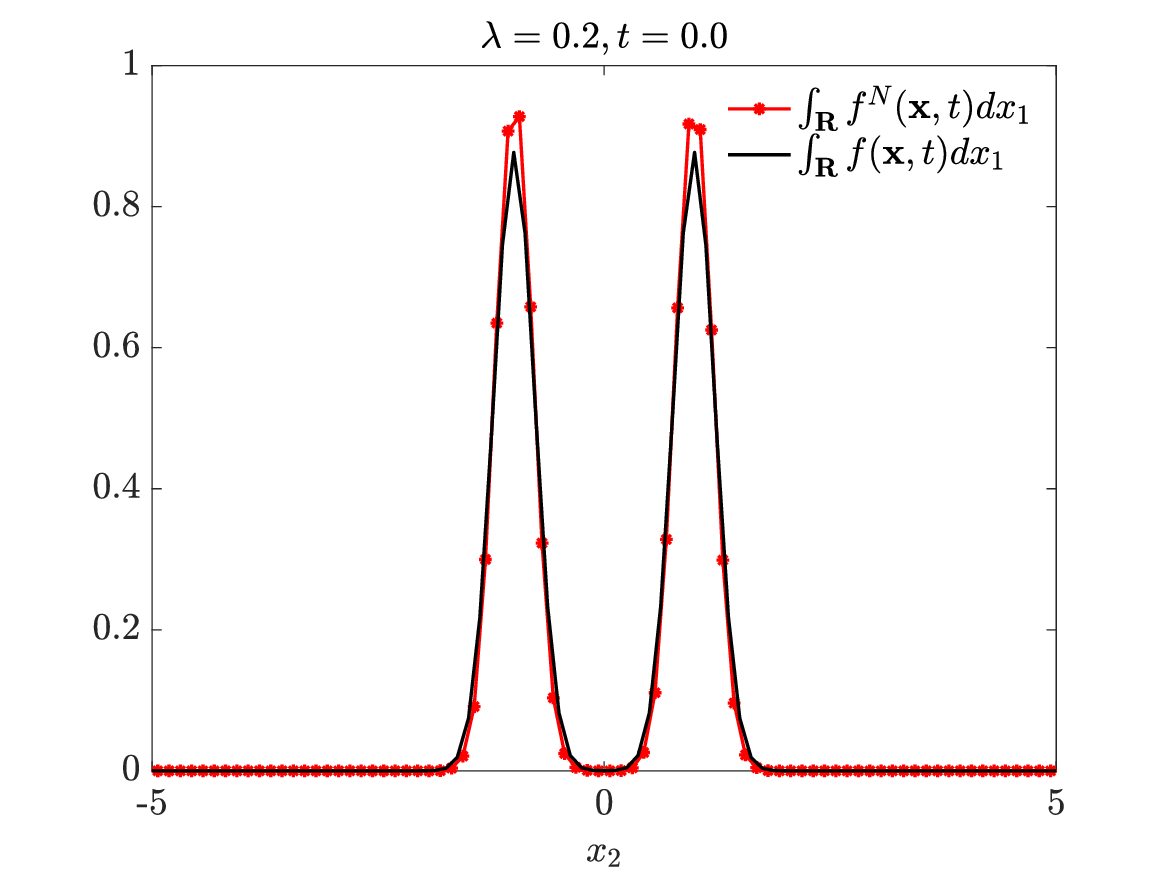}
\includegraphics[scale = 0.2]{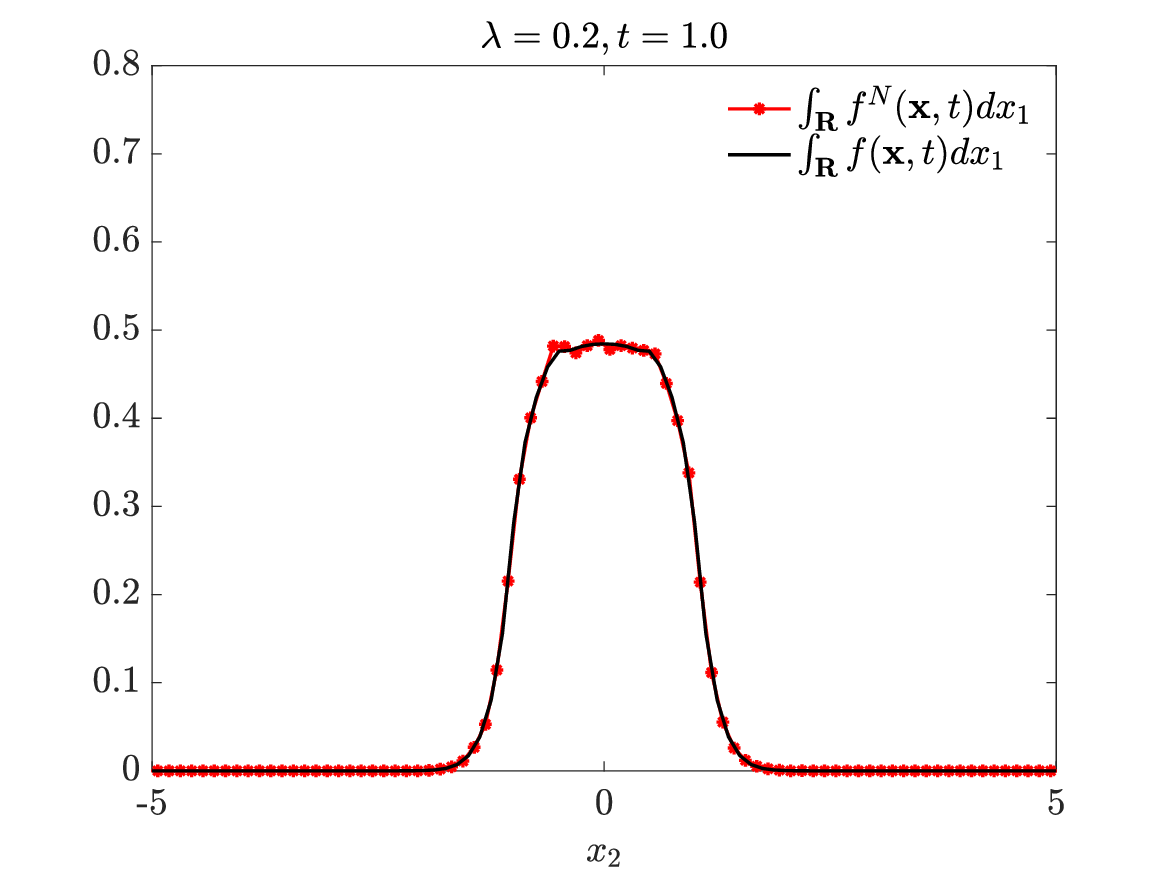}
\includegraphics[scale = 0.2]{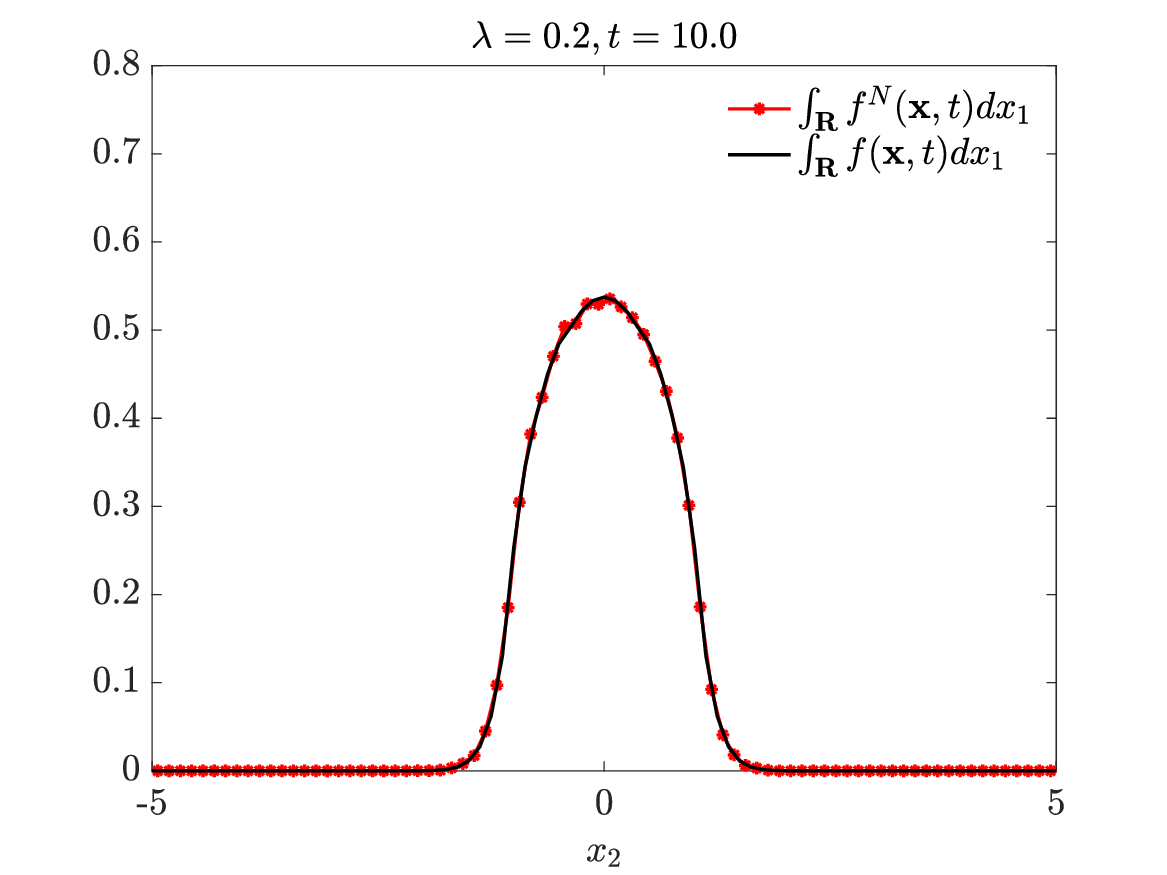}
\caption{\textbf{Test 1c}. Top row: evolution of the particles' distribution $f^N(\x,t)$ at times $t = 0,1,10$ obtained from \eqref{eq:part1} with $P\equiv1$.  Second row: evolution o the numerical solution of the Fokker-Planck equation \eqref{eq:FP_P1} over the same grid. Bottom row: evolution of the marginal densities of $f^N(\x,t)$ and $f(\x,t)$. We considered as target domain $D = \{\x \in \mathbb R^2: |\x-\x_0|\le 1\}$, $\x_0= (0,0)$, $\sigma^2 = 0.2$ and $\lambda = 0.2$, the nonconstant diffusion function $\kappa(\x,\tilde\x_0)$ has been defined in \eqref{eq:FP_simp_diffusion}. We introduced a grid of $N_x = 81$ gridpoints in $[-5,5]$,  time discretization of $[0,10]$ with $\Delta t = 10^{-2}$. Initial condition given in \eqref{eq:init2D}. }
\label{fig:2D_1}
\end{figure}

In this test we compare the evolution of the reconstructed density $f^N(\x,t)$  of the particles' system defined by $\{\x_i\}_{i=1}^N$ solution to \eqref{eq:part1} in the 2D case to the numerical solution of Fokker-Planck model \eqref{eq:FP}. We consider as initial distribution a sum of four Gaussian densities
\begin{equation}
\label{eq:init2D}
f(\x,0) = \sum_{k=1}^4 \dfrac{c_k}{2\pi\sigma_0^2}\exp\left\{ - \dfrac{1}{2\sigma_0^2} \left( (x_1-m_{x,k})^2 + (x_2-m_{y,k})\right)\right\}
\end{equation}
with $m_{x,1} = -m_{x,2} = m_{x,3} = -m_{x,4} = 1$ and $m_{y,1} = -m_{y,2} = -m_{y,3} = m_{y,4} = -1$, $\sigma_0^2 = 0.2$ and $c_1 = c_2 = 3/8$, $c_3 = c_4 = 1/8$. We fixed as a target domain $D = \{\x \in \mathbb R^2: |\x-\x_0|\le 1\}$ with $\x_0 = (0,0)$ and $N = 10^4$ particles whose dynamics, given by \eqref{eq:part1}, has been integrated over the time interval $[0,10]$ with an Euler-Maruyama scheme and $\Delta t =10^{-2}$. The particles' distribution $f^N(\x,t)$ have been reconstructed through standard 2D histograms over the interval $[-5,5]$ discretized with $N_x = 81$ gridpoints. Over the introduced grids in space and time, and starting from the initial distribution \eqref{eq:init2D}, we solved the Fokker-Planck model. The initial positions $\x_i(0)$ of the particles are sampled from \eqref{eq:init2D}. 

\begin{figure}
\centering
\includegraphics[scale = 0.20]{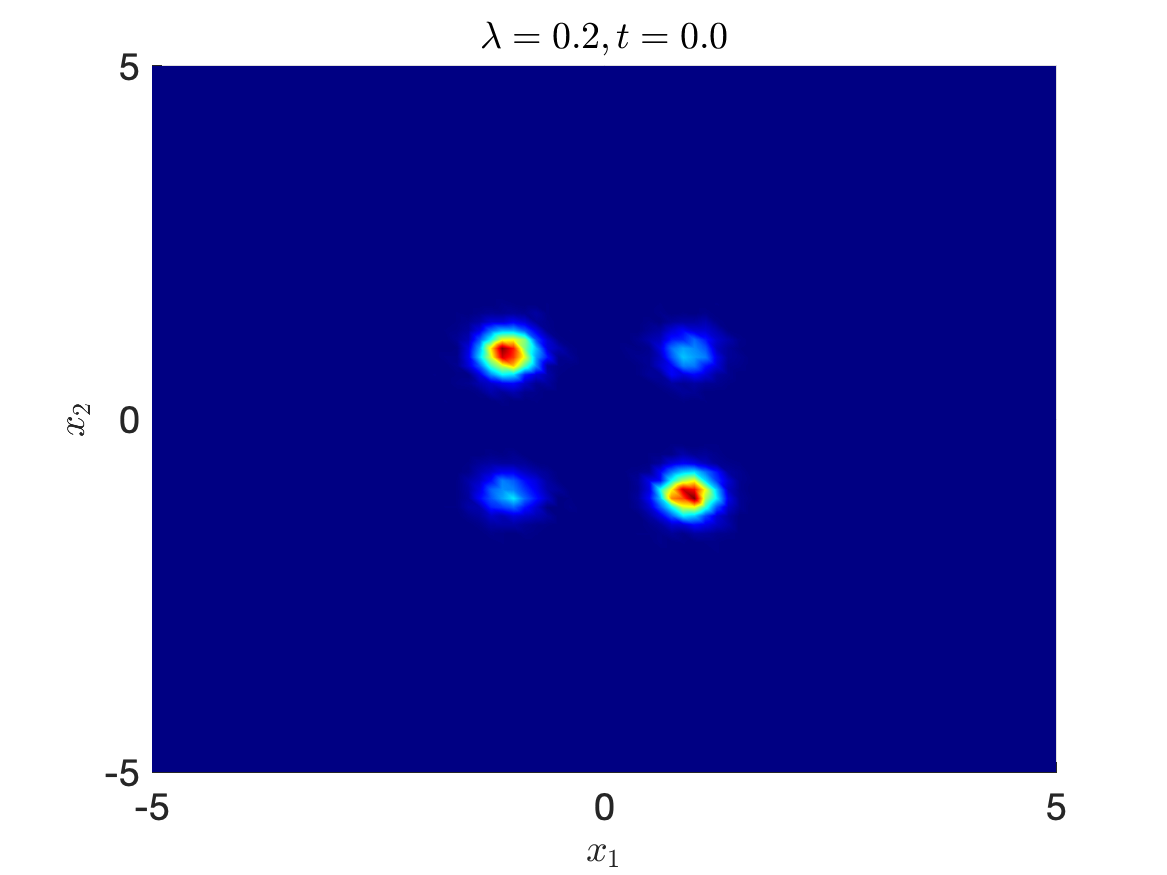}
\includegraphics[scale = 0.20]{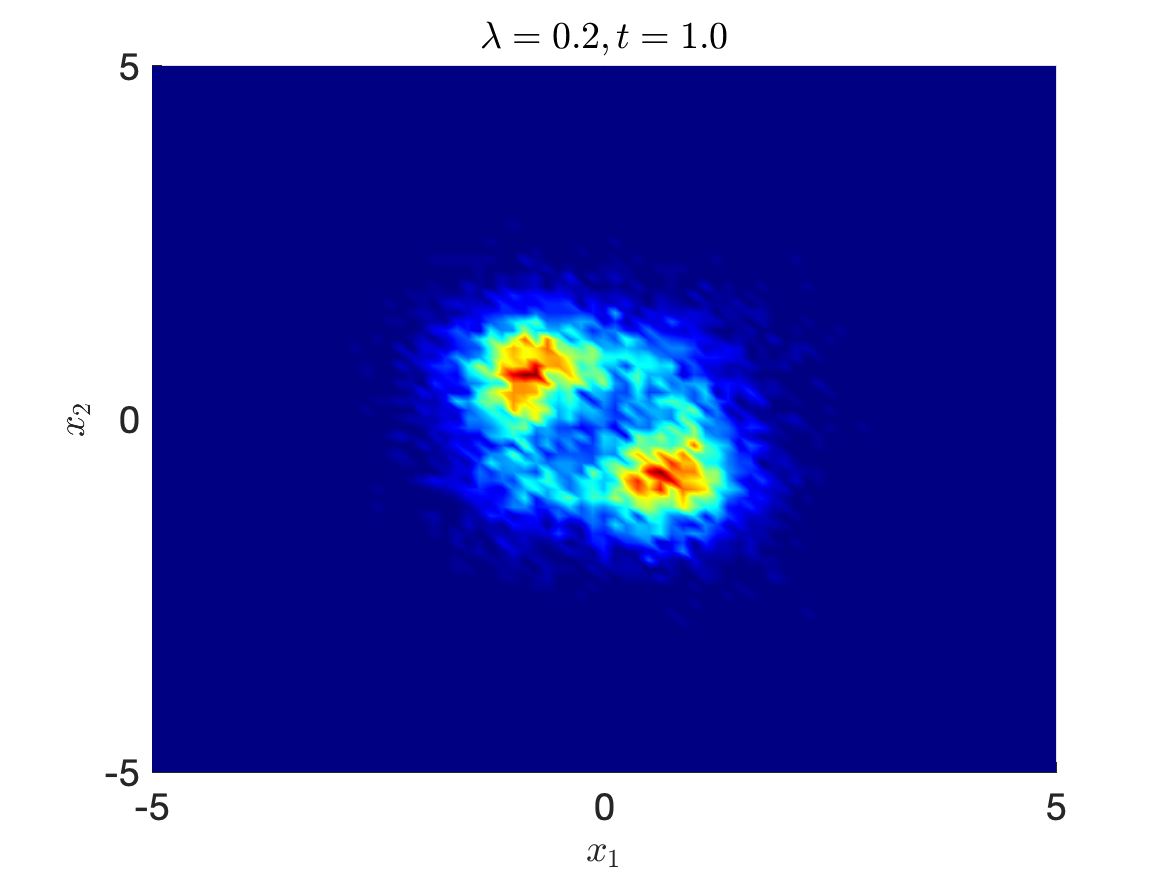}
\includegraphics[scale = 0.20]{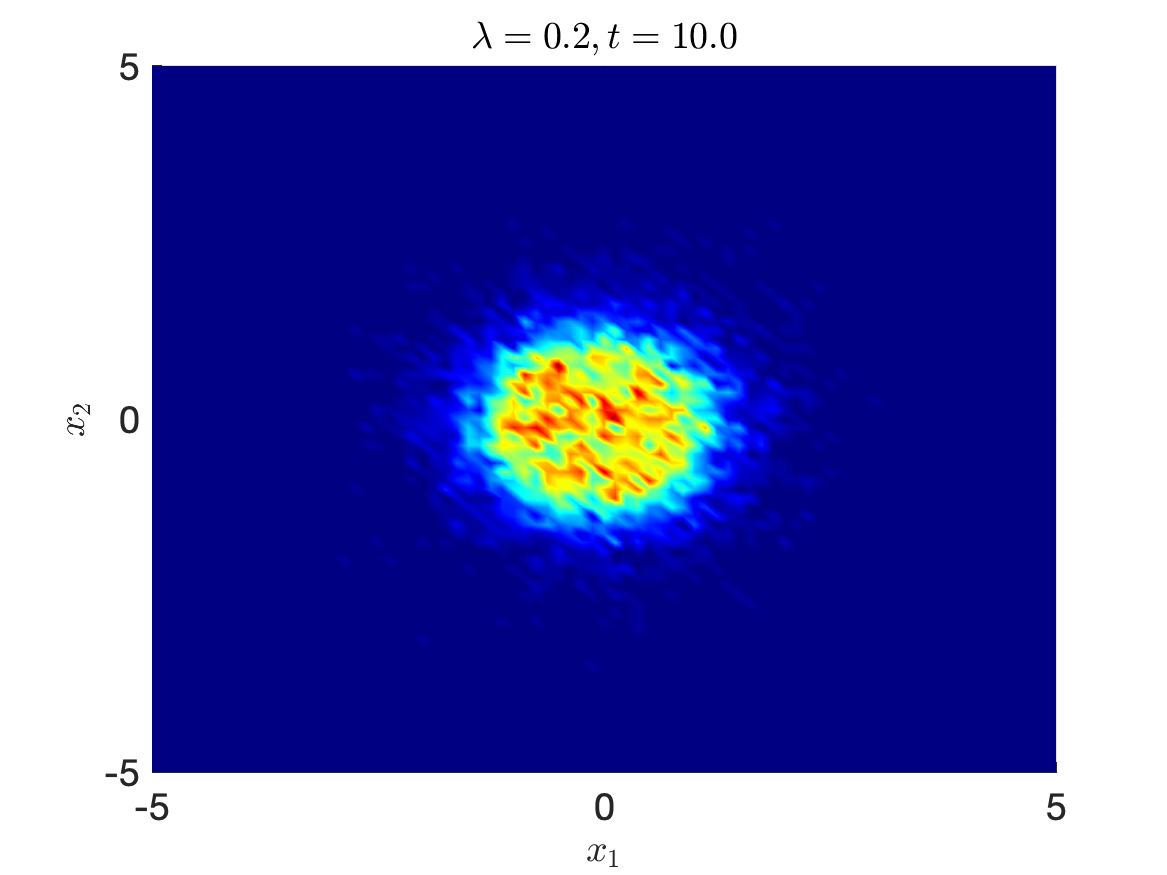}\\
\includegraphics[scale = 0.2]{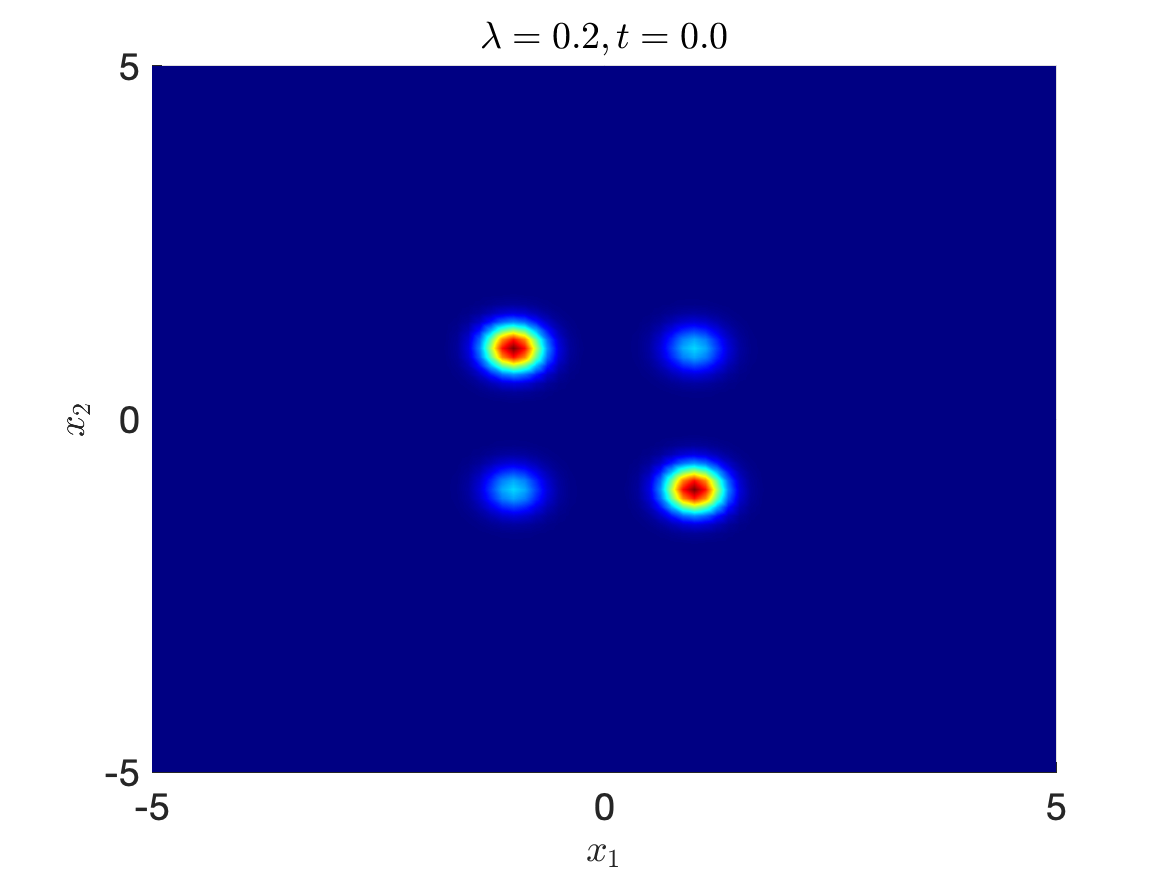}
\includegraphics[scale = 0.2]{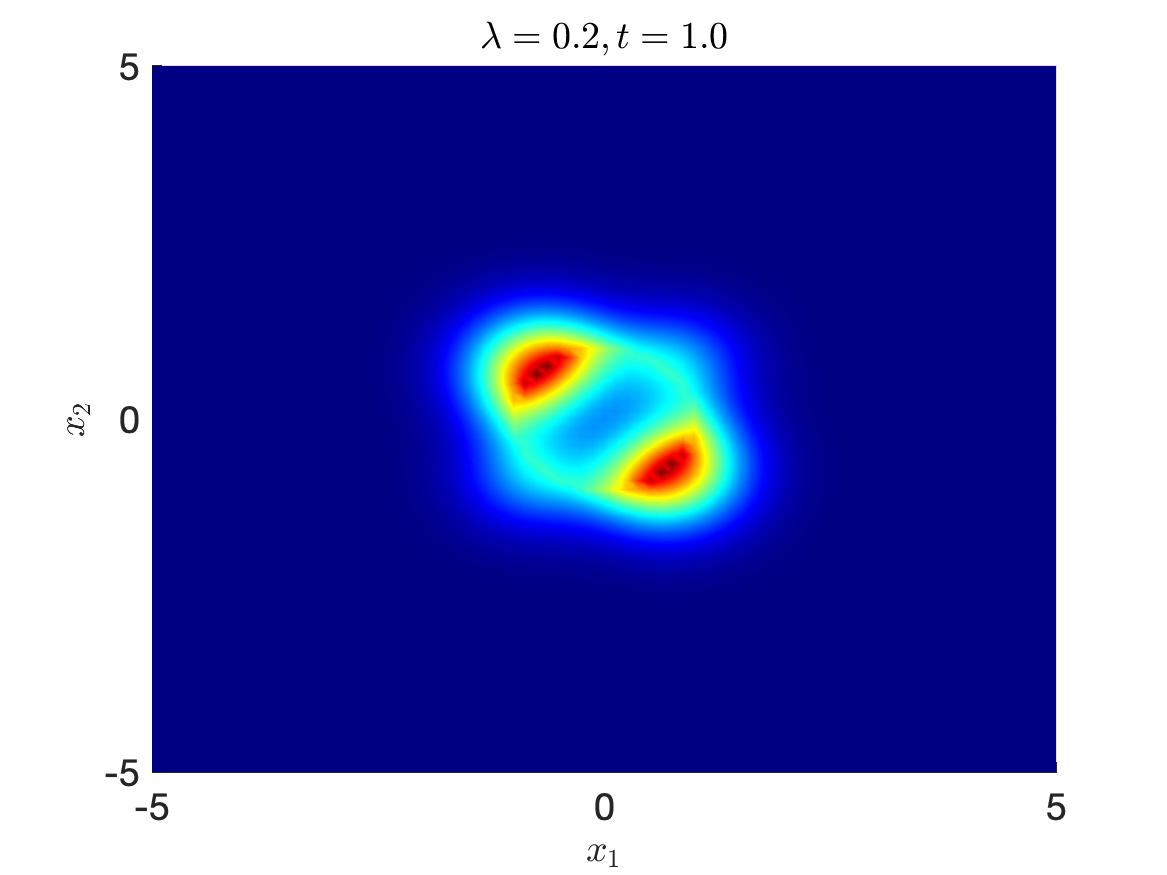}
\includegraphics[scale = 0.2]{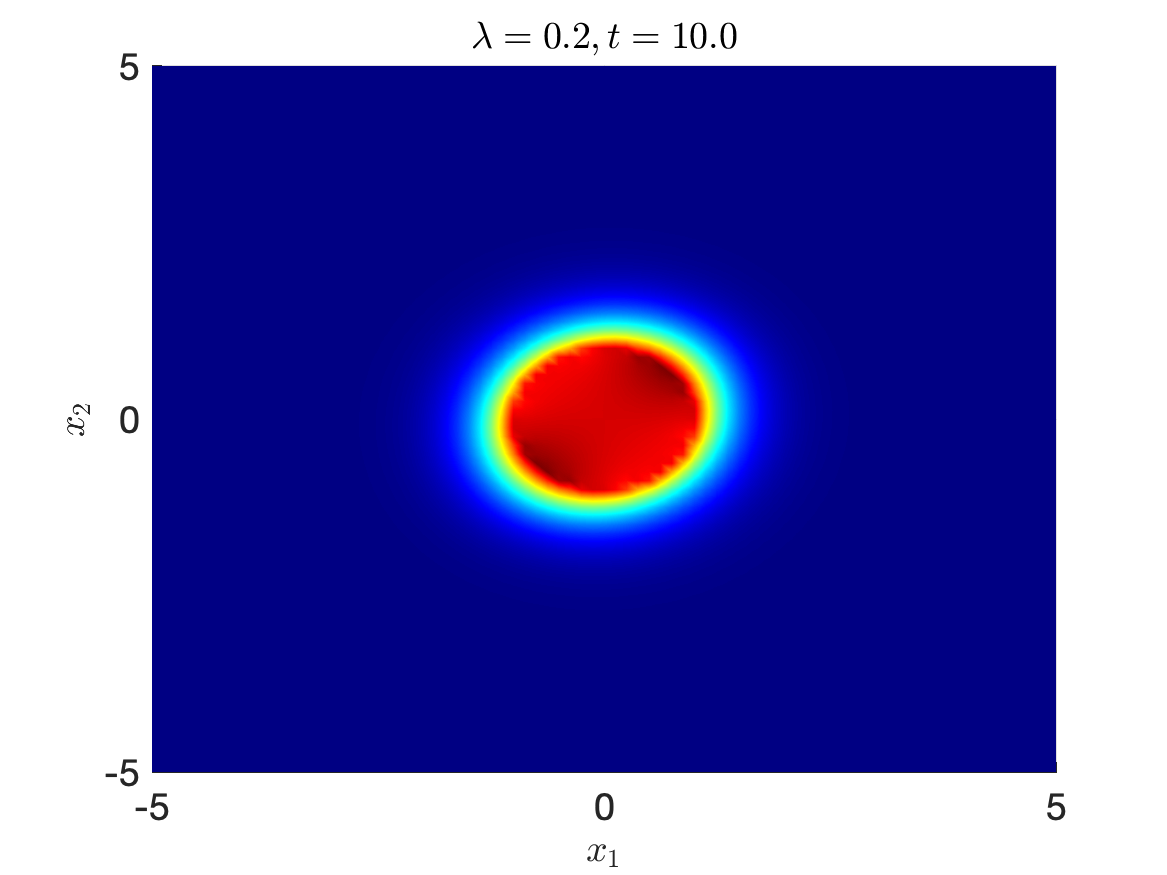}\\
\includegraphics[scale = 0.2]{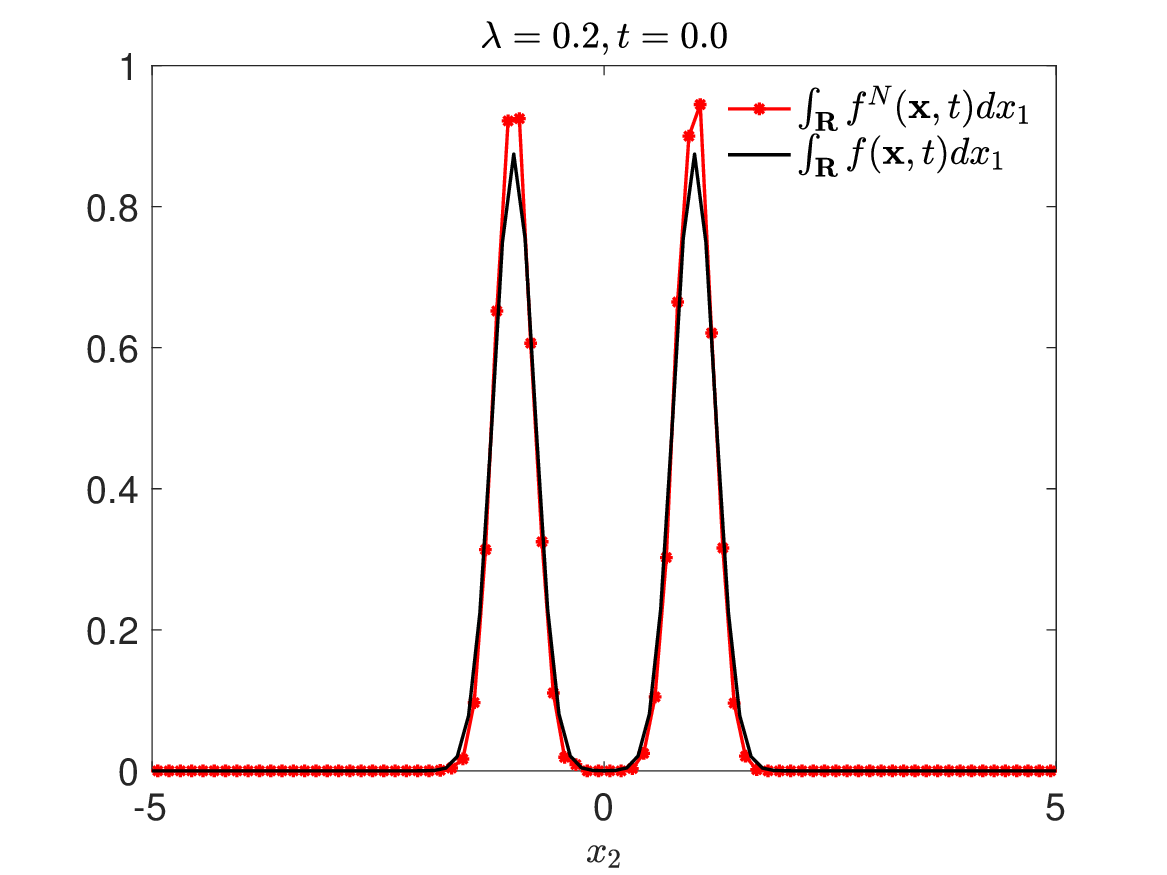}
\includegraphics[scale = 0.2]{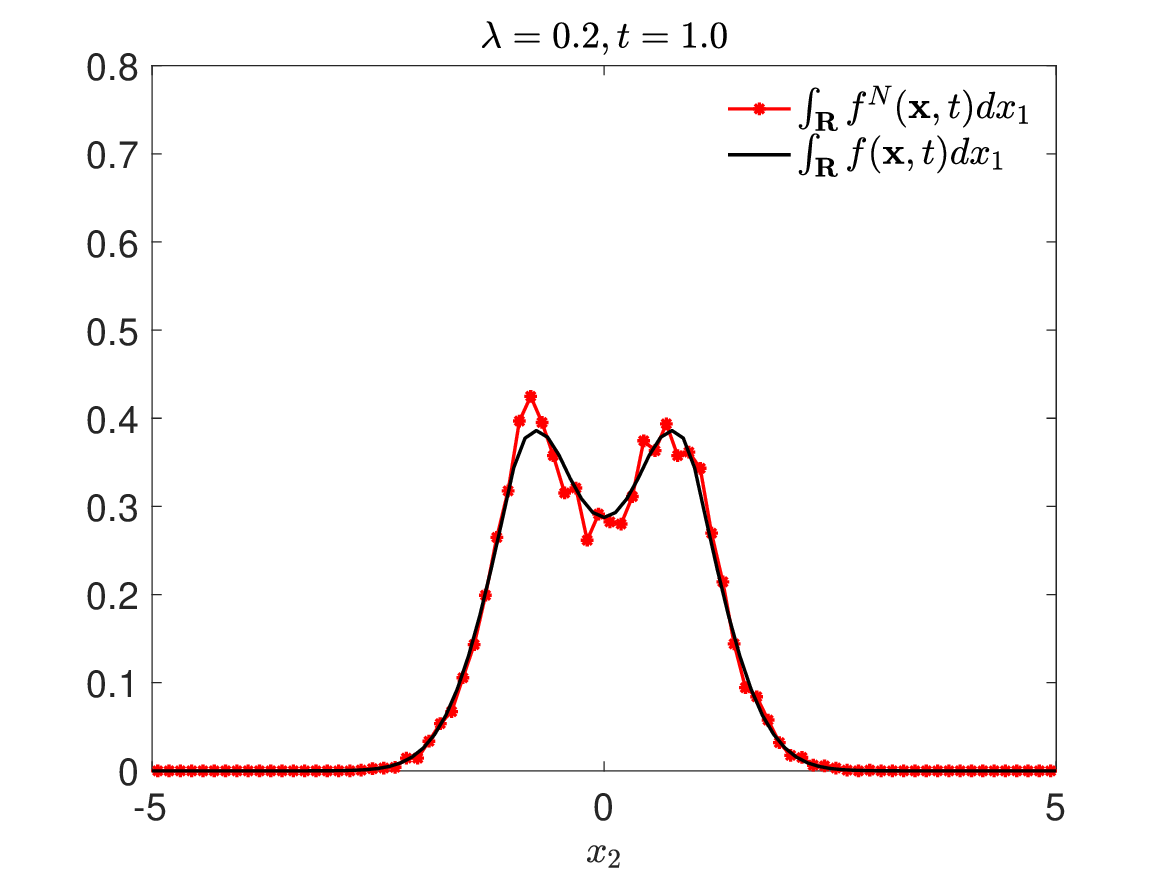}
\includegraphics[scale = 0.2]{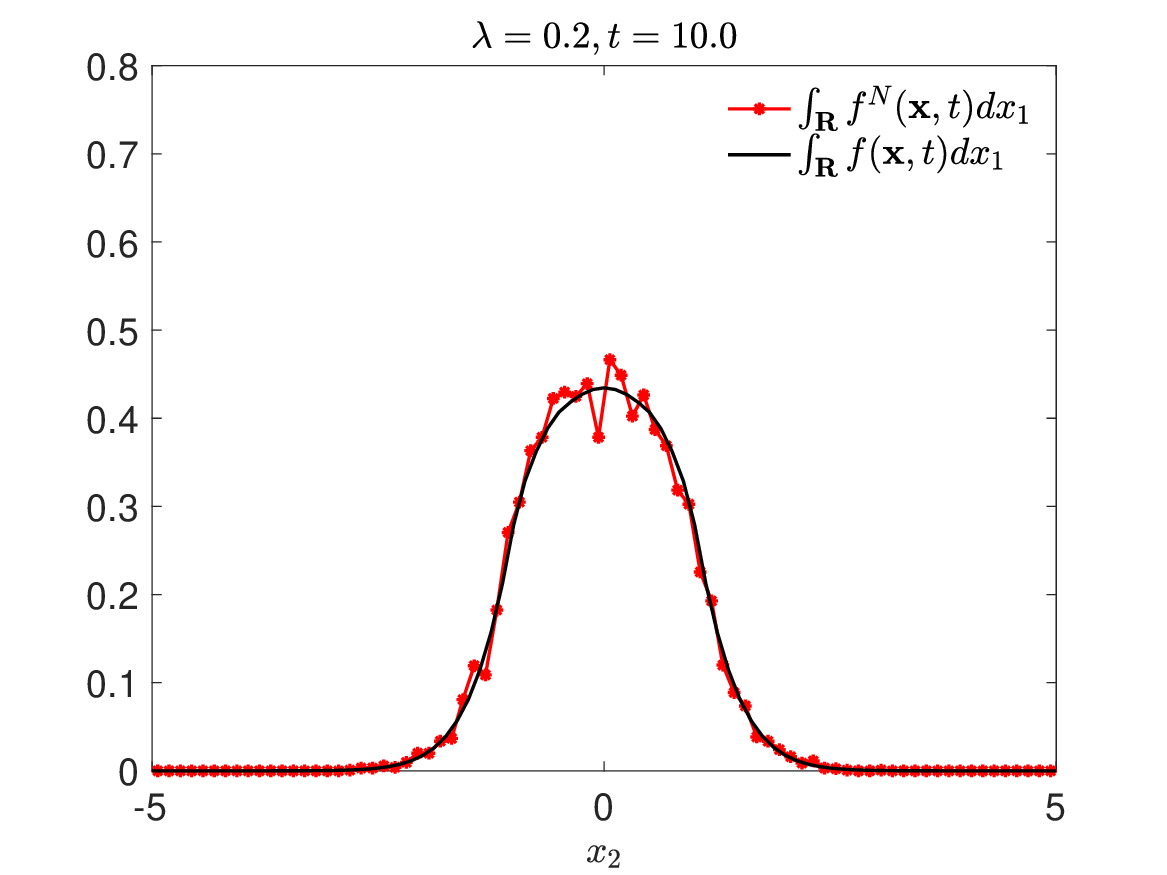}
\caption{\textbf{Test 1c}. Top row: evolution of the particles' distribution $f^N(\x,t)$ at times $t = 0,1,10$ obtained from \eqref{eq:part1} with $P(\x,\mathbf y)$ in \eqref{eq:CS} and $N = 10^4$ particles.  Second row: evolution o the numerical solution of the Fokker-Planck equation \eqref{eq:FP} over the same grid. Bottom row: evolution of the marginal densities of $f^N(\x,t)$ and $f(\x,t)$. We considered as target domain $D = \{\x \in \mathbb R^2: |\x-\x_0|\le 1\}$, $\x_0= (0,0)$, $\sigma^2 = 0.2$ and $\lambda = 0.2$. We introduced a grid of  $N_x = 81$ gridpoints in $[-5,5]$,  time discretization of $[0,10]$ with $\Delta t = 10^{-2}$. Initial condition given in \eqref{eq:init2D}.   }
\label{fig:2D_2}
\end{figure}

In the top row of Figure \ref{fig:2D_1} we report the evolutions of the reconstructed distribution $f^N(\x,t)$ of the particles' system \eqref{eq:part1} and, in the second row, of the numerical solution $f(\x,t)$ to the Fokker-Planck equation with nonconstant diffusion obtained in the uniform interction case $P\equiv 1$ \eqref{eq:FP_P1}. We depict the 2D distributions at times $t = 0,1,10$ and for fixed $\lambda = 0.2$. To better compare the results, in the bottom row, we show the agreement between the two marginals from which we may observe good agreement also in  2D  for uniform interactions. 


Finally, in the top row of Figure \ref{fig:2D_2} we report the evolutions of $f^N(\x,t)$ obtained from the particles' system \eqref{eq:part1} in the case $P(\x,\mathbf y)$ defined in \eqref{eq:CS}. In the second row, we report the numerical solution $f(\x,t)$ to the Fokker-Planck equation with nonuniform interactions  \eqref{eq:FP}. We depict the 2D distributions at times $t = 0,1,10$ and for fixed $\lambda = 0.2$. As before, in the bottom row, we show  the marginal distributions from which we may observe good agreement also in the case of nonconstant interaction forces. 


\subsection{Trends to equilibrium}

We compute equilibration rates for the introduced Fokker-Planck models \eqref{eq:FP}. In all the subsequent tests we fix as target domain $D = \{\x \in \mathbb R^d: |\x-\x_0\le 1|\}$, for simplicity we will  fix $\x_0$ to be the null vector in $\mathbb R^d$. We recall that the relative Shannon entropy is defined as follows
\begin{equation}
\label{eq:Htest}
H(f|f^\infty)(t) = \int_{\mathbb R^d} f(\x,t) \log \dfrac{f(\x,t)}{f^\infty(\x)}d\x.
\end{equation}
In the following we focus on the case $d = 1$, for which we showed entropic decay, and on the case $d = 2$, for which we will present computational results. In particular, we will adopt the We recall that the decay of the Shannon entropy has been proven for the Fokker-Planck model with nonconstant diffusion \eqref{eq:FP_P1}.

\begin{figure}
\centering
\includegraphics[scale = 0.3]{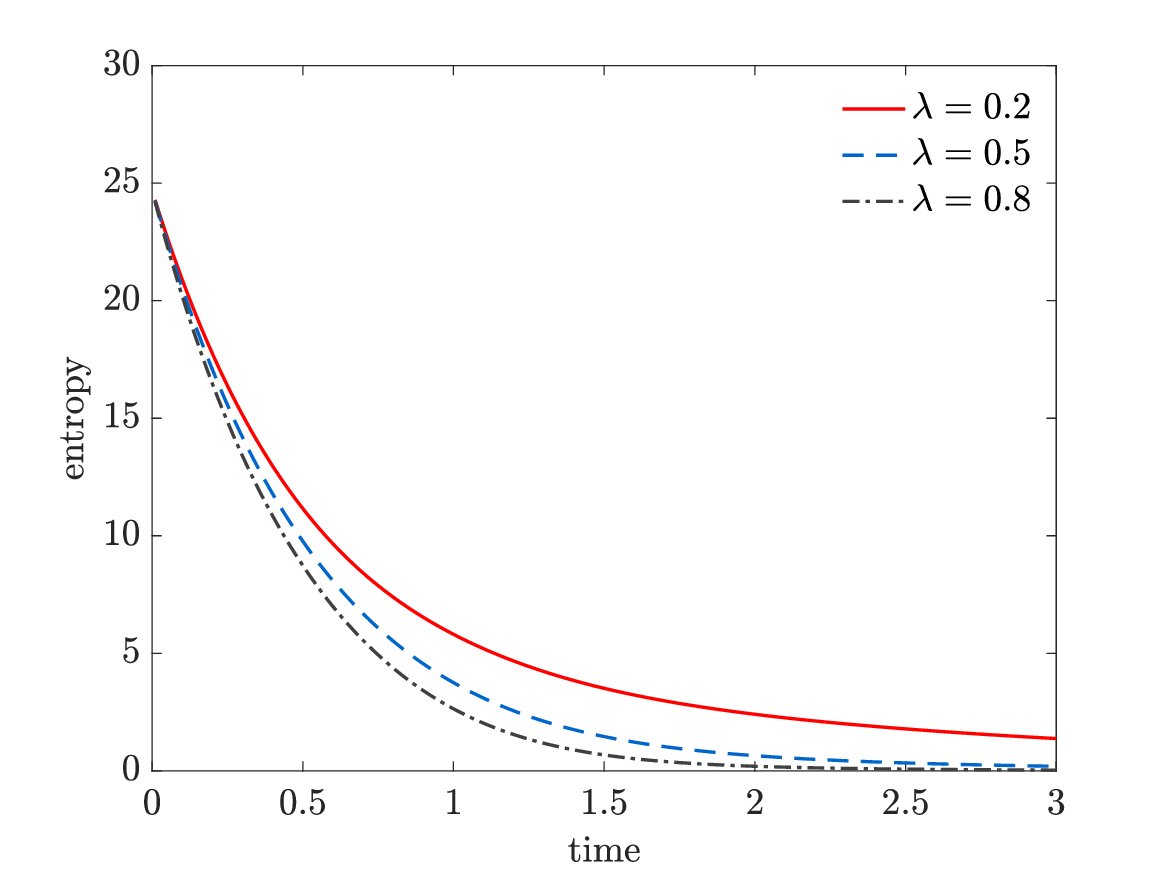}
\includegraphics[scale = 0.3]{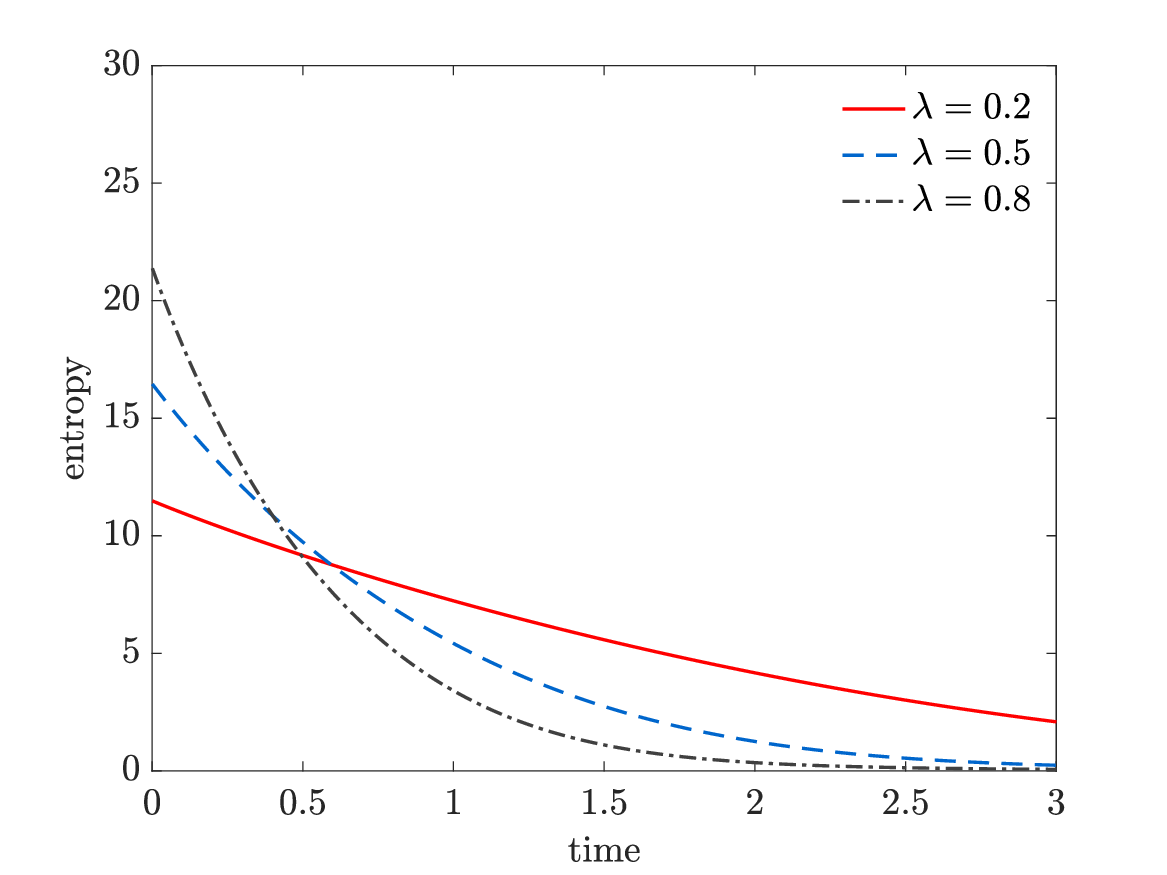}
\caption{\textbf{Test 2}. Left: evolution of the relative entropy functional $H(f|f^\infty)(t)$ obtained from \eqref{eq:FP_P1}, $d=1$, and analytical equilibrium $f^\infty(\x)$ defined in  \eqref{eq:finfty_1}. Right: evolution of the entropy functional $H(f|f^\textrm{ref}_\lambda)(t)$ for the Fokker-Planck equation \eqref{eq:FP}, $d=1$, with space-dependent $P(x,y)$ defined in \eqref{eq:CS}. The reference solution $f^\textrm{ref}_\lambda(x,T)$ have been obtained for a discrertization of $[-5,5]$ with $N_x=801$ gridpoints and $T = 50$. In both cases $\Delta t = \frac{\Delta x^2}{10}$ and the initial distribution is \eqref{eq:f0_test2}.  }
\label{fig:evo1D}
\end{figure}

We consider first the case $d = 1$ and we introduce the initial distribution 
\begin{equation}
\label{eq:f0_test21}
f_0(x) = \dfrac{3}{4\sqrt{2\pi\sigma_0^2}}\exp\left\{-\dfrac{|x+2|^2}{2\sigma_0^2}\right\} + \dfrac{1}{4\sqrt{2\pi\sigma_0^2}}\exp\left\{-\dfrac{|x-2|^2}{2\sigma_0^2}\right\}, 
\end{equation}
with  $x\in \mathbb R$, $\sigma_0^2 = \frac{1}{20}$. 

We numerically approximate the uniform interaction model \eqref{eq:FP_P1} over the interval $[-5,5]$ discretized by $N_x = 81$ gridpoints and over the time interval $[0,3]$ with $\Delta t = \frac{\Delta x^2}{10}$. The numerical integration has been performed with RK4 scheme. In \eqref{eq:Htest} we considered the analytical equilibrium defined by \eqref{eq:finfty_1}.  In the left plot of Figure \ref{fig:evo1D} we report the evolution of $H(f|f^\infty)$ for several $\lambda = 0.2,0.5,0.8$. We can observe that low values of  $\lambda \in (0,1]$ trigger slow convergence rates towards the analytical equilibrium \eqref{eq:finfty_1}. 

Furthermore, we evaluate numerically the convergence of the model \eqref{eq:FP} with Cucker-Smale-type interaction forces $P(x,y)$ in \eqref{eq:CS}. In this case, we remark that it is very difficult to obtain analytically the equilibrium distribution $f^\infty(x)$, which is then replaced with a reference large time solution $f^{\textrm{ref}}_\lambda(x,T)$, depending on the parameter $\lambda$, and obtained from the integration of \eqref{eq:FP} over $[0,T]$, $T = 50$, over the interval $[-5,5]$ with a discretization obtained with $N_x = 801$ gridpoints. In the right plot of Figure \ref{fig:evo1D} we report the evolution of $H(f|f^{\textrm{ref}}_\lambda)(t)$ where the approximation of $f(x,t)$ is considered on a more coarse grid with $N_x = 81$ gridpoints. Also in this case, for small $\lambda \in (0,1]$, the swarm is partially informed on the position of $x_0 \in  D$ whereas each agent senses all the other agents of the swarm in terms of the their relative distance. We can observe that the obtained Fokker-Planck equation still converges in time but at a lower rate. 

In the case $d =2$ we consier the initial distribution 
\begin{equation}
\label{eq:f0_test2}
\begin{split}
f_0(\x) =& \dfrac{3}{8\pi\sigma_0^2}\exp\left\{-\dfrac{|x_1-2|^2}{2\sigma_0^2} - \dfrac{|x_2+2|^2}{2\sigma_0^2} \right\} \\
&+ \dfrac{1}{8\pi\sigma_0^2}\exp\left\{-\dfrac{|x_1-2|^2}{2\sigma_0^2} -\dfrac{|x_2-2|^2}{2\sigma_0^2} \right\}, 
\end{split}
\end{equation}
with  $\x = (x_1,x_2)\in \mathbb R^2$, $\sigma_0^2 = \frac{1}{20}$. The target domain is $D = \{\x \in \mathbb R^2: |\x-\x_0|\le 1\}$, $\x_0 = (0,0)$. We numerically approximate the uniform interaction model \eqref{eq:FP_P1} over $[-5,5] \times [-5,5]$ discretized by $N_x = 81$ gridpoints in each space directions and over the time interval $[0,T]$, $T = 3$, with $\Delta t = \frac{\Delta x^2}{10}$. The numerical integration has been performed with RK4 scheme. In \eqref{eq:Htest} we considered the analytical equilibrium defined by \eqref{eq:finfty_1}.  In the left plot of Figure \eqref{fig:evo2D} we report the evolution of $H(f|f^\infty)(t)$ for several $\lambda = 0.2,0.5,0.8$. In 2D we may observe that only for large times the rate of convergence towards equilibrium is affected by the value of $\lambda$. 

For model with nonlocal interactions \eqref{eq:FP} with $P(\x,\mathbf y)$ as in \eqref{eq:CS} we compute the evolution of the relative entropy with respect to the reference large time solution $f^{\textrm{ref}}_\lambda(\x,T)$ with $T = 20$ computed through the integration of the model with a RK4 scheme over $[-5,5]\times [-5,5]$ and $N_x = 81$ in each space direction, $\Delta t = \frac{\Delta x^2}{10}$. We may observe that the evolution of $H(f|f^{\textrm{ref}}_{\lambda})(t)$ is still monotone decreasing and depends on the value of $\lambda \in (0,1]$. 

\begin{figure}
\centering
\includegraphics[scale = 0.3]{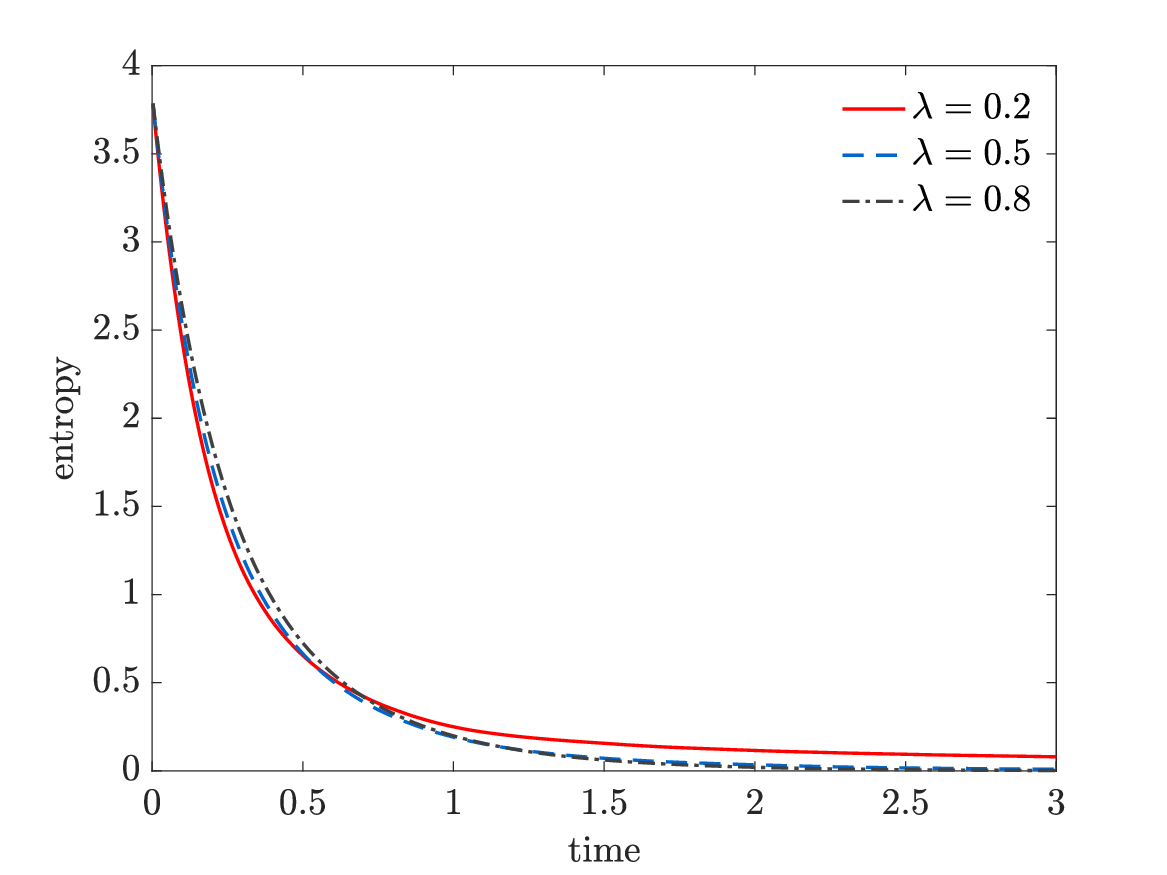}
\includegraphics[scale = 0.3]{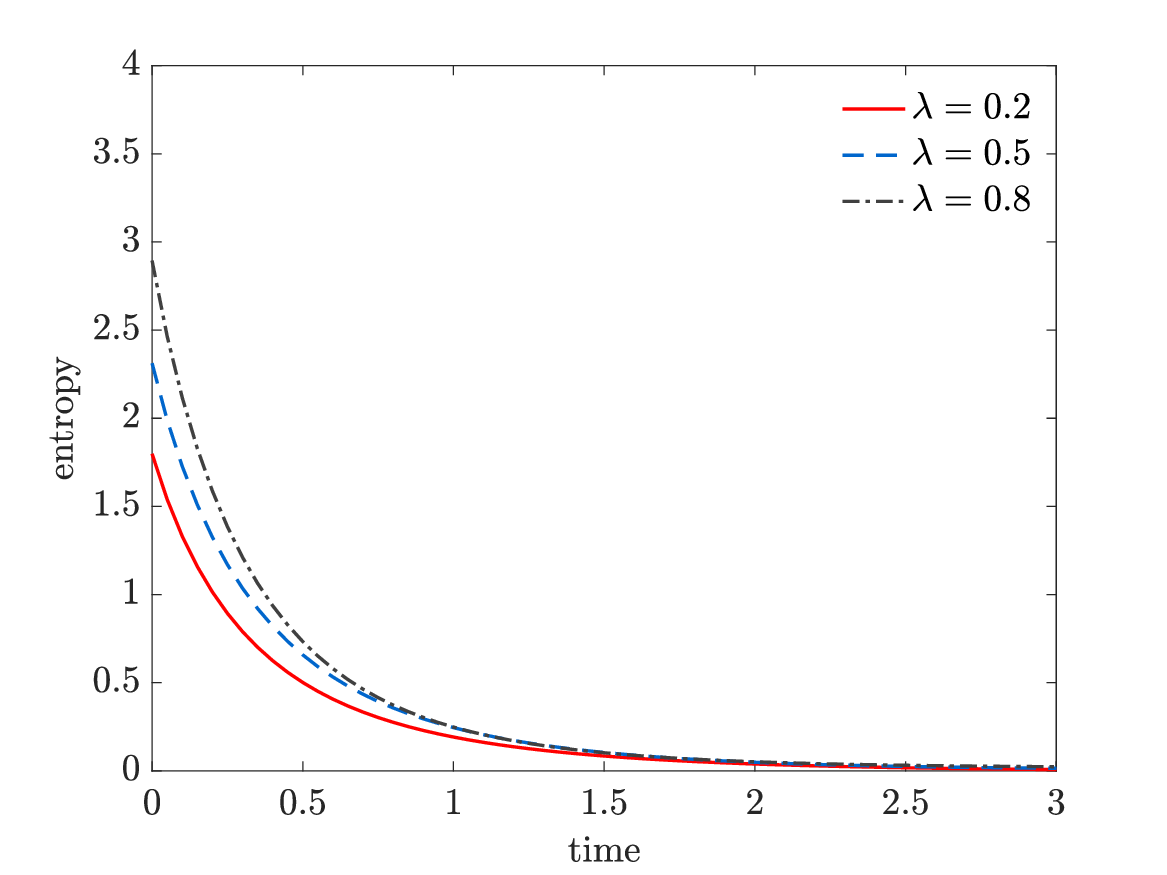}
\caption{\textbf{Test 2}. Left: evolution of the relative entropy functional $H(f|f^\infty)(t)$ obtained from \eqref{eq:FP_P1}, $d=2$, and analytical equilibrium $f^\infty(\x)$ defined in  \eqref{eq:finfty_1}. Right: evolution of the entropy functional $H(f|f^\textrm{ref}_\lambda)(t)$ for the Fokker-Planck equation \eqref{eq:FP}, $d=1$, with space-dependent $P(\x,\mathbf{y})$ defined in \eqref{eq:CS}. The reference solution $f^\textrm{ref}_\lambda(\x,T)$ have been obtained for a discrertization of $[-5,5] \times [-5,5]$ with $N_x=81$ gridpoints in each space direction and $T = 50$. In both cases $\Delta t = \frac{\Delta x^2}{10}$ and the initial distribution is \eqref{eq:f0_test2}.  }
\label{fig:evo2D}
\end{figure}

\section*{Conclusions}
In this paper, we investigated the large time behavior of a system of interacting particles  modeling the relaxation of a large swarm of robots, whose task is to cover uniformly a portion of a domain in $\R^d$, feeling each other in terms of their distance. The task has been modeled by a Fokker-Planck-type model with a linear drift and a time and position dependent diffusion coefficient,  which possesses a steady state distribution explicitly computable. For this new nonlocal Fokker--Planck equation, existence, uniqueness and positivity of a global solution have been proven in any dimension of the space,  with precise equilibration rates of the solution towards its quasi-stationary distribution in the one-dimensional case.  Numerical simulations then show that the swarm converges to the right equilibrium also in dimension $d=2$ and for a communication function dependent on the relative position of the agents. This suggests that the Fokker-Planck model is well-posed even in higher dimension of the space variable.  Extensions of the modelling approach to include sub-critical confinement and dynamics on manifolds are actually under study and will be presented elsewhere. 

\section*{Acknowledgements}
This work has been written within the activities of the GNFM group of INdAM (National Institute of High Mathematics). M.Z. acknowledges partial support of MUR-PRIN2020 Project No. 2020JLWP23.


\begin{thebibliography}{99}

\bibitem{Ac}
E. Ackerman. Mobile Robots Cooperate to 3D Print Large Structures. \emph{IEEE Spectrum: Technology, Engineering, and Science News}, 28 August 2018.

\bibitem{Ahn_etal}
H. Ahn, S.-Y. Ha, D. Kim, F. W. Schlöder, W. Shim. The mean-field limit of the Cucker-Smale model on complete Riemannian manifolds. \emph{Quart. Appl. Math.}, \textbf{80}(3):403--450, 2022. 

\bibitem{Ahn2}
H. Ahn, J. Byeon, S.-Y. Ha, J. Yoon. Asymptotic tracking of a point cloud moving on Riemannian manifolds. \emph{SIAM J. Contr. Optim}, in press. 
%
\bibitem{AP}
G. Albi, L. Pareschi. Modeling of self-organized systems interacting with a few individuals: From microscopic to macroscopic dynamics. \emph{Appl. Math. Lett.}, \textbf{4}:397--401, 2013. 


\bibitem{Arkeryd}
L. Arkeryd. On the Boltzmann equation. Part I: Existence. \emph{Arch. Ration. Mech. Anal.}, \textbf{45}:1--16, 1972.

\bibitem{A}
F. Auricchio. A continuous model for the simulation of manufacturing swarm robotics. \emph{Comput. Mech.}, 2022. 

\bibitem{ATZ}
F. Auricchio, G. Toscani, M. Zanella. Fokker-Planck modeling of many-agent systems in swarm manufacturing: asymptotic analysis and numerical results. \emph{Commun. Math. Sci.}, in press. 

\bibitem{ATZ2}
F. Auricchio, G. Toscani, M. Zanella. Trends to equilibrium for a nonlocal Fokker-Planck equation. \emph{Applied Math. Letters}, 145: 108746, 2023. 

\bibitem{BCC}
F. Bolley, J. Cañizo, J. A. Carrillo. Stochastic mean-field limit: non-Lipschitz forces and swarming. \emph{Math. Mod. Meth. Appl. Sci.}, \textbf{21}:2179--2210, 2011. 


\bibitem{BGM}
F. Bolley, A. Guillin and F. Malrieu. Trend to equilibrium and particle ap- proximation for a weakly selfconsistent Vlasov-Fokker-Planck equation. \emph{ESAIM: Math. Model. Numer. Anal.}, \textbf{44}(5):867--884, 2010.

\bibitem{CFPT}
M. Caponigro, M. Fornasier, B. Piccoli, E. Trélat. Sparse stabilization and optimal control of the Cucker-Smale model. \emph{Math. Control Relat. Fields}, \textbf{3}(4):447--466, 2013.

\bibitem{CFRT}
J.A. Carrillo, M. Fornasier, J. Rosado, G. Toscani. Asymptotic flocking dynamics for the kinetic Cucker–Smale model. \emph{SIAM J. Math. Anal.}, \textbf{42}:218--236, 2010. 

\bibitem{CFTV}
J. A. Carrillo, M. Fornasier, G. Toscani, F. Vecil. Particle, kinetic, and hydrodynamic models of swarming. In \emph{Mathematical Modeling of Collective Behavior in Socio-Economic and Life Sciences},  G. Naldi, L. Pareschi, G. Toscani, (eds). Modeling and Simulation in Science, Engineering and Technology, Birkh{\"a}user Boston, 2010. 

\bibitem{Cavanna}
A. Cavagna, A. Cimarelli, I. Giardina, G. Parisi, R. Santagati, F. Stefanini, M. Viale. Scale-free correlations in starling flocks. \emph{Proc. Natl. Acad. Sci. U.S.A.}, \textbf{107}(26):11865--11870, 2010. 


\bibitem{CKPP}
Y.-P. Choi, D. Kalise, J. Peszek, A. A. Peters. A collisionless singular Cucker-Smale model with decentralized formation control. \emph{SIAM J. Appl. Dyn. Systems}, \textbf{18}(4):1954--1981, 2019. 

\bibitem{COT}
Y.-P. Choi, D. Oh, O. Tse. Controlled pattern formation of stochastic Cucker-Smale systems with network structures. \emph{Commun. Nonlinear Sci. Numer. Simul.}, \textbf{111}:106474, 2022. 


\bibitem{CH}
N. Correll, H. Hamann. Probabilistic modeling of swarming systems. In \emph{Springer Handbook of Computational Intelligence}, pp. 1423--1432, Springer, Berlin, Heidelberg, 2015. 

\bibitem{C}
I. D. Couzin, J. Krause, R. James, G. D. Ruxton, N. R. Franks. Collective memory and spatial sorting in animal groups. \emph{J. Theor. Biol.}, \textbf{218}:1--11, 2002. 

\bibitem{CS}
F. Cucker, S. Smale. Emergent behavior in flocks. \emph{IEEE Trans. Automat. Control}, \textbf{52}:852--862, 2007.

\bibitem{DM}
P. Degond, S. Motsch. Continuum limit of self-driven particles with orientation interaction. \emph{Math. Mod. Meth. Appl. Scie.}, \textbf{18}(supp01):1193--1215, 2008.

\bibitem{DOB}
M.R. D’Orsogna, Y.L. Chuang, A.L. Bertozzi, L.S. Chayes. Self-propelled particles with soft-core interactions: patterns, stability, and collapse. \emph{Phys. Rev. Lett.}, \textbf{96}:104-302, 2006. 

\bibitem{D_etal}
S. Duncan, G. Estrada-Rodriguez, J. Stocek, M. Dragone, P. Vargas,  H. Gimperlein. Efficient quantitative assessment of robot swarms: coverage and targeting Levy strategies. \emph{Bioinspir. Biomim.}, \textbf{17}(3), 2022.



\bibitem{FPTT}
G. Furioli, A. Pulvirenti, E. Terraneo, G. Toscani. Fokker-Planck equations in the modelling of socio-economic phenomena. \emph{Math. Models Methods Appl. Scie.}, \textbf{27}(1):115-158, 2017. 


\bibitem{HJKPZ}
S.-Y. Ha, J. Jung, J. Kim, J. Park, X. Zhang. Emergent behaviors of the swarmalator model for position-phase aggregation. \emph{Math. Mod. Meth. Appl. Sci.}, \textbf{29}(12):2225--2269, 2019. 

\bibitem{HKSF}
S.-Y. Ha, D. Kim, F. W. Schlöder. Emergent behaviors of Cucker-Smale flocks on Riemannian  manifolds. \emph{IEEE Trans. Automat. Control}, \textbf{66}(7):3020--3035, 2021. 

\bibitem{HT}
S.-Y. Ha, E. Tadmor. From particle to kinetic and hydrodynamic descriptions of flocking. \emph{Kinet. Relat. Models}, \textbf{1}:415--435, 2008.

\bibitem{H}
H. Hammann. \emph{Swarm Robotics: A Formal Approach}, Springer Cham, 2018.

\bibitem{HW}
H. Hamann, H. W{\"o}rn. A framework of space–time continuous models for algorithm design in swarm robotics. \emph{Swarm Intelligence}, \textbf{2}(2): 209--239, 2008.

\bibitem{K_etal}
A. J. King, S. J. Portugal, D. Strömbom, R. P. Mann, J. A. Carrillo, D. Kalise, G. de Croon, H. Barnett, P. Scerri, R. Gro\ss, D. R. Chadwick, M. Papadopoulou. Biologically inspired herding of animal groups by robots. \emph{Math. Ecol. Evol}, 00, 1--9, 2023. 

\bibitem{LLB}
C. Le Bris, P.-L. Lions. Existence and uniqueness of solutions to Fokker-Planck type equations with irregular coefficients. \emph{Commun. Partial Differ. Equ.}, \textbf{33}(7):1272--1317, 2008. 


\bibitem{LZ}
N. Loy, M. Zanella. Structure preserving schemes for Fokker-Planck equations with nonconstant diffusion matrices. \emph{Math. Comput. Simul.}, \textbf{188}: 342-362, 2021. 

\bibitem{M}
S. Méléard. Asymptotic behaviour of some interacting particle systems; McKean- Vlasov and Boltzmann models. In \emph{Probabilistic models for nonlinear partial differential equations (Montecatini Terme, 1995)}, Lecture Notes in Mathematics. Springer, Berlin, 1996.

\bibitem{MT}
S. Motsch, E. Tadmor. Heterophilious dynamics enhances consensus. \emph{SIAM Rev.}, \textbf{56}(4):577--621, 2014. 

\bibitem{OV}
 F. Otto,  C. Villani. Generalization of an inequality by Talagrand and links with the logarithmic Sobolev inequality. \emph{J. Funct. Anal.} \textbf{173} 361--400, 2000.

\bibitem{Ox}
N. Oxman, J. Duro-Royo, S. Keating, B. Peters, E. Tsai. Towards robotic swarm printing. \emph{Architectural Design}, \textbf{84}(3):108--115, 2014.

\bibitem{PT}
L. Pareschi, G. Toscani. \emph{Interacting Multiagent Systems: Kinetic Equations \& Monte Carlo Methods}, Oxford University Press, 2013. 

\bibitem{PZ}
L. Pareschi, M. Zanella.  Structure preserving schemes for nonlinear Fokker-Planck equations and applications. \emph{J. Sci. Comput.}, \textbf{74}(3): 1575-1600, 2018.


\bibitem{temam}
R. Temam. Sur la résolution exacte et approchée d'un problème hyperbolique non linéaire de T. Carleman. \emph{Arch. Ration. Mech. Anal.}, \textbf{35}:351--362,1969.

\bibitem{T1}
G. Toscani. Entropy dissipation and the rate of convergence to equilibrium for the Fokker-Planck equation. \emph{Quart. Appl. Math.}, \textbf{LVII}:521--541, 1999. 

\bibitem{TV}
G. Toscani, C. Villani. On the trend to equilibrium for some dissipative systems with slowly increasing a priori bounds. \emph{J. Statist. Phys.}, \textbf{98}(5--6):1279--1309, 2000. 

\bibitem{TZ}
G. Toscani, M. Zanella. On a class of Fokker-Planck equations with subcritical confinement. \emph{Rend. Lincei Mat. Appl.}, \textbf{32}:471--497,2021. 


\bibitem{V}
T. Vicsek, A. Czir\'ok, E. Ben-Jacob, I. Cohen, O. Shochet. Novel type of phase transition in a system of self-driven particles. \emph{Phys. Rev. E}, \textbf{75}(6):1226--1229, 1995. 
%
%

\end{thebibliography}
\end{document}